\documentclass[11pt,twoside]{amsart}

\usepackage{amssymb}
\usepackage{amsthm}
\usepackage{amsmath}
\usepackage{tikz}
\usepackage{ifthen}
\usepackage{enumerate}
\usetikzlibrary{arrows}
\usetikzlibrary{shapes,snakes}

\usepackage{hyperref}

\tikzstyle{block} = [draw,fill=blue!20,minimum size=0.5em]

\tikzstyle{branch}=[fill,shape=circle,minimum size=3pt,inner sep=0pt]

\usepackage{pgf,pgfarrows,pgfnodes,pgfautomata,pgfheaps}
\usepackage[latin1]{inputenc}

\newtheorem{thm}{Theorem}

\newtheorem{lem}[thm]{Lemma}
\newtheorem{prop}[thm]{Proposition}
\newtheorem{cor}[thm]{Corollary}

\newcommand{\ignore}[1]{}
\newcommand{\set}[1]{\left\{#1\right\}}
\newcommand{\supp}[1]{\mbox{\rm{supp}}#1}

\def\ba{\begin{array}}
\def\ea{\end{array}}
\def\beq{\begin{equation}}
\def\eeq{\end{equation}}
\def\bea{\begin{eqnarray}}
\def\eea{\end{eqnarray}}
\def\beann{\begin{eqnarray*}}
\def\eeann{\end{eqnarray*}}
\def\tn{\textnormal}
\def\A{\mathcal{A}}
\def\F{\mathcal{F}}
\def\L{\mathcal{L}}
\def\M{\mathcal{M}}

\def\O{\mathcal{O}}

\def\tn{\textnormal}

\def\es{\emptyset}

\def\sm{\setminus}

\def\het{\hat}
\def\a{\alpha}
\def\b{\beta}
\def\r{\rho}
\def\t{\theta}

\def\l{\ell}

\def\mR{\mathbb{R}}

\def\mS{\mathbb{S}}

\DeclareMathOperator{\SA}{SA}
\DeclareMathOperator{\LS}{LS}
\DeclareMathOperator{\BZ}{BZ}
\DeclareMathOperator{\Las}{Las}
\DeclareMathOperator{\BCC}{BCC}
\DeclareMathOperator{\conv}{conv}

\DeclareMathOperator{\Diag}{Diag}

\DeclareMathOperator{\FRAC}{FRAC}

\title[Elementary polytopes with high lift-and-project ranks]{Elementary polytopes with high lift-and-project ranks for strong positive semidefinite operators}
\thanks{Some of the material in this manuscript appeared in a preliminary
form in Au's PhD Thesis~\cite{Au14a}.}

\author{Yu Hin (Gary) Au}
\thanks{Yu Hin (Gary) Au: Research of this author was supported in part by a Tutte
Scholarship, a Sinclair Scholarship, an NSERC scholarship, research grants
from University of Waterloo and Discovery Grants from NSERC. Department of Mathematics, Milwaukee School of Engineering, Milwaukee, Wisconsin, U.S.A. E-mail: au@msoe.edu}

\author{Levent Tun\c{c}el}
\thanks{Levent Tun\c{c}el: Research of this author was supported in part by research grants
from University of Waterloo, Discovery Grants from NSERC and U.S. Office of Naval Research under award numbers: N00014-12-1-0049 and N00014-15-1-2171.
Department of Combinatorics and Optimization, Faculty of Mathematics, University of Waterloo, Waterloo, Ontario, N2L 3G1 Canada. E-mail: ltuncel@uwaterloo.ca}

\date{\today}


\begin{document}

\maketitle              

\begin{abstract}
We consider operators acting on convex subsets of the unit hypercube. These operators are used in constructing convex relaxations of combinatorial optimization problems presented as a 0,1 integer programming problem or a 0,1 polynomial
optimization problem.  Our focus is mostly on operators that, when expressed as a lift-and-project operator, involve the use of semidefiniteness constraints in the lifted space, including operators due to Lasserre and variants of the Sherali--Adams and Bienstock--Zuckerberg operators. We study the performance of these semidefinite-optimization-based lift-and-project
operators on some elementary polytopes --- hypercubes that are chipped (at least one vertex of the hypercube removed by intersection with a closed halfspace) or cropped (all $2^n$ vertices of the hypercube removed
by intersection with $2^n$ closed halfspaces) to varying degrees of severity $\rho$. We prove bounds on $\rho$ where these operators would perform badly on the aforementioned examples. We also show that the integrality gap of the chipped hypercube is invariant under the application of several lift-and-project operators of varying strengths.
\end{abstract}

\section{Introduction}

A foundational approach to tackling combinatorial optimization problems is to start with a 0,1 integer programming formulation 
and construct convex relaxations of the feasible
region which leads to a tractable (whether in practice or theory, of course hopefully in both)
optimization problem with essentially the same linear objective function but a convex feasible region.
Let $P \subseteq [0,1]^n$ denote  the feasible region of the linear programming relaxation of an
initial 0,1 integer programming problem.  In our convex relaxation approach, we are hoping
to construct a tractable representation of the convex hull of integer points in $P$, i.e., the  \emph{integer hull of $P$}
\[
P_I := \tn{conv}\left(P \cap \set{0,1}^n\right).
\]
However, it is impossible to efficiently find a tractable description of $P_I$ for a general $P$ (unless $\mathcal{P} = \mathcal{NP}$).  So, in many cases we may have to be content with tractable convex relaxations that are not exact (strict supersets of the integer hull of $P$).

\emph{Lift-and-project} methods provide an organized way of generating a sequence of convex relaxations of $P$
which converge to the integer hull $P_I$ of $P$ in at most $n$ rounds. Minimum number of rounds required to obtain the integer hull by a lift-and-project operator $\Gamma$ is called the \emph{$\Gamma$-rank} of $P$. Computational success of lift-and-project methods on some combinatorial optimization problems and various applications is relatively well-documented
(starting with the theoretical foundations in Balas' work in the 1970's \cite{Balas1974}; appeared as \cite{Balas98a}), and the majority of these computational successes come from lift-and-project methods which generate polyhedral relaxations. While many lift-and-project methods utilize in addition positive semidefiniteness constraints which in theory help generate tighter relaxations of $P_I$, the underlying convex optimization problems require significantly more computational resources and are prone to run into more serious numerical stability issues.  Therefore, before committing to the usage of a certain lift-and-project method, it would be wise to understand the conditions under which the usage of additional computational resources would be well justified.  Indeed, this argument applies to any collection of lift-and-project operators that trade off quality of approximation with computational resources (time, memory, etc.) required. That is, to utilize the strongest operators, one needs a better understanding of the class of problems on which these strongest operators' computational demands will be worthwhile in the returns they provide.

In the next section, we introduce a number of known lift-and-project operators and some of their basic properties, with the focus being on the following operators
(every one of these utilizes positive semidefiniteness constraints):
\begin{itemize}
\item
$\SA_+$ (see \cite{Au14a,AuT16a}), a positive semidefinite variant of the Sherali--Adams operator $\SA$ defined in~\cite{SheraliA90a};
\item
$\Las$, due to Lasserre~\cite{Lasserre01a};
\item
$\BZ_+'$ (see \cite{Au14a,AuT16a}), a strengthened version of the Bienstock--Zuckerberg operator $\BZ_+$~\cite{BienstockZ04a}.
\end{itemize}

Then, in Section 3, we look into  some elementary polytopes which represent some basic situations in 0,1 integer programs. We consider two families of polytopes: unit hypercubes that are \emph{chipped} or \emph{cropped} to various degrees of severity. First, given an integer $n \geq 1$ and a real number $\r$ where $0 \leq \r \leq n$, the \emph{chipped hypercube} is defined to be
\[
P_{n,\r} := \set{ x \in [0,1]^n : \sum_{i=1}^n x_i  \leq n- \r}.
\]
Similarly, we define the \emph{cropped hypercube}
\[
Q_{n,\r} := \set{x \in [0,1]^n : \sum_{i \in S}(1-x_i) + \sum_{i \not\in S}  x_i \geq \r,~\forall S \subseteq [n]},
\]
where $[n]$ denotes the set $\set{1,\ldots,n}$. These two families of polytopes had been shown to be bad instances for many lift-and-project methods and cutting-plane procedures (see, among others,~\cite{ChvatalCH89a, CornuejolsL01a, CookD01a, GoemansT01a, Laurent03a, Cheung07a} and more recently~\cite{KurpiszLM15a}). Moreover, these elementary sets are interesting in many other contexts as well. For instance, note that each constraint defining $Q_{n,\r}$ removes a specific extreme point of the unit hypercube from the feasible region.  In many 0,1 integer programming problems and in 0,1 mixed integer programming problems, such \emph{exclusion} constraints are relatively commonly used.

Herein, we show that these sets are also bad instances for the strongest known operators, extending the previously known results in this vein. In particular, we show the following:
\begin{itemize}
\item
The $\SA_+$-rank of $P_{n,\r}$ is $n$ for all $\r \in (0,1)$, and is at most $n - \lceil \r \rceil + 1$ for all $\r \in (0,n)$. In contrast, we show that $\tilde{\LS}$ (a simple polyhedral operator defined in~\cite{GoemansT01a} that is similar to the $\LS_0$ operator due to Lov\'asz and Schrijver~\cite{LovaszS91a}) requires $n$ iterations to return the integer hull of $P_{n,\r}$ for all non-integer $\r \in (0,n-1)$.
\item
The integrality gap of $\SA_+^k \left( P_{n,\r} \right)$ in the direction of the all-ones vector is
\[
1 + \frac{(n-k)(1-\r)}{(n-1)(n-k+k\r)}
\]
for all $n \geq 2, k \in \set{0,1,\ldots, n}$, and $\r \in (0,1)$. Moreover, we show that this integrality gap is exactly the same, if we replace $\SA_+$ by an operator as weak as $\tilde{\LS}$.
\item
The $\Las$-rank of $P_{n,\r}$ is $n$ for all $\r \in \left( 0 , \frac{n^2-1}{2n^{n+1} - n^2-1 }  \right]$. This strengthens earlier work by Cheung~\cite{Cheung07a}, who showed the existence of such a positive $\r$ but did not give concrete bounds.
\item
The $\Las$-rank of $Q_{n,\r}$ is $n$ for all $\r \in \left( 0, \frac{n+1}{2^{n+2} - n-3}  \right)$, and at most $n-1$ for all $\r > \frac{n}{2^{n+1}-2}$.
\item
There exist $n, \r$ where the $\BZ_+'$-rank of $P_{n,\r}$ is $\Omega(\sqrt{n})$, providing what we believe to be the first example where $\BZ_+'$ (and as a consequence, the weaker $\BZ_+$) requires more than a constant number of iterations to return the integer hull of a set.
\end{itemize}

The tools we use in our analysis, which involve zeta and moment matrices, build on earlier work by others (such as~\cite{Laurent03a} and~\cite{Cheung07a}), and could be useful in analyzing lift-and-project relaxations of other sets. Finally, we conclude the manuscript by noting some interesting behaviour of the integrality gaps of some lift-and-project relaxations.

We remark that preliminary and weaker versions of our results on the Lasserre relaxations of $P_{n,\r}$ and $Q_{n,\r}$ were published in the first author's PhD thesis~\cite{Au14a}. During the writing of this manuscript, we discovered that Kurpisz, Lepp{\"a}nen and Mastrolilli~\cite{KurpiszLM15a} had obtained similar and stronger results. In fact, in their work, they characterized general conditions for when the $(n-1)^{\tn{th}}$ Lasserre relaxation is not the integer hull. Using very similar ideas to theirs, we have subsequently sharpened our results to those appearing in this manuscript.

\section{Preliminaries}

In this section, we establish some notation and describe several lift-and-project operators utilizing positive semidefiniteness constraints.

\subsection{The operators $\LS_+$ and $\SA_+$}

First, let $\F$ denote $\set{0,1}^n$, and define $\mathcal{A} := 2^{\F}$, the power set of $\F$. As shown in~\cite{Zuckerberg03a}, many existing lift-and-project operators can be seen as lifting a given relaxation $P$ to a set of matrices whose rows and columns are indexed by sets in $\A$. For more motivation and details on this framework, the reader may refer to~\cite{AuT16a}.

We first define the operator $\SA_+$, which can be interpreted as a strengthened variant of the Sherali--Adams operator~\cite{SheraliA90a}. Given $P \subseteq [0,1]^n$, define the cone
\[
K(P) := \set{\begin{pmatrix}\lambda \\ \lambda x \end{pmatrix} \in \mathbb{R}^{n+1}: \lambda \geq 0,~x \in P},
\]
where we shall denote the extra coordinate by $0$. Next, we introduce a family of sets in $\A$ that are used extensively by the operators we will introduce in this paper. Given a set of indices $S \subseteq [n]$ and $t \in \set{0,1}$, we define
\[
S |_t := \set{ x \in \F : x_i = t,~\forall i \in S}.
\]
Note that $\es|_0 = \es|_1 = \F$. Also, to reduce cluttering, we write $i|_t$ instead of $\set{i}|_t$. Next, given any integer $\l \in \set{0,1,\ldots,n}$, we define $\A_{\l} := \set{S|_1 \cap T|_0 : S,T \subseteq [n], S \cap T = \es, |S| + |T| \leq \l}$ and $\A_{\l}^+ := \set{S|_1 : S \subseteq [n],  |S| \leq \l}$. For instance,
\[
\A_{1} = \set{\F, 1|_1, 2|_1, \ldots, n|_1, 1|_0, 2|_0, \ldots, n|_0},
\]
while
\[
\A_{1}^+ = \set{\F, 1|_1, 2|_1, \ldots, n|_1}.
\]
Given any vector $y \in \mathbb{R}^{\A'}$ for some $\A' \subseteq \A$ which contains $\F$ and $i|_1$ for all $i \in [n]$, we let $\het{x}(y) := (y_{\F}, y_{1|_1}, \ldots, y_{n|_1})^{\top}$. Sometimes we may also alternatively index the entries of $\het{x}(y)$ as $(y_0, y_1, \ldots, y_n)^{\top}$, when we verify these vectors' membership in $K(P)$.

Finally, let $\mathbb{S}_+^{n}$ denote the set of $n$-by-$n$ real, symmetric matrices that are positive semidefinite, and let $e_i$ denote the $i^{\tn{th}}$ unit vector (of appropriate size, which will be clear from the context). Then, given any positive integer $k$, we define the operator $\SA_+^k$ as follows:

\begin{enumerate}
\item
Let $\widehat{\SA}_+^k(P)$ be the set of matrices $Y \in \mathbb{S}_+^{\A_{k}}$ which satisfy all of the following conditions:
\begin{itemize}
\item[($\SA_+ 1$)]
$Y[\F, \F] = 1$.
\item[($\SA_+ 2$)]
For every $\a \in \A_k$:
\begin{itemize}
\item[(i)]
$\het{x}(Ye_{\a}) \in K(P)$;
\item[(ii)]
$Ye_{\a} \geq 0$.
\end{itemize}
\item[($\SA_+ 3$)]
For every $S|_1 \cap T|_0 \in \A_{k-1}$,
\[
Ye_{S|_1 \cap T|_0 \cap j|_1} + Ye_{S|_1 \cap T|_0 \cap j|_0} = Ye_{S|_1 \cap T|_0}, \quad \forall j \in [n] \sm (S \cup T).
\]
\item[($\SA_+4$)]
For all $\a,\b \in \A_k$ such that $\a \cap \b = \es, Y[\a,\b] = 0$.
\item[($\SA_+5$)]
For all $\a_1,\a_2, \b_1, \b_2 \in \A_k$ such that $\a_1 \cap \b_1 = \a_2 \cap \b_2, Y[\a_1, \b_1] = Y[\a_2, \b_2]$.
\end{itemize}
\item
Define
\[
\SA_+^k(P) := \set{x \in \mathbb{R}^n: \exists Y \in \widehat{\SA}_+^k(P),  \het{x}(Ye_{\F})= \begin{pmatrix} 1 \\x \end{pmatrix} }.
\]
\end{enumerate}

The $\SA_+^{k}$ operator extends the lifted space of the original level-$k$ Sherali--Adams operator $\SA^k$ (which are matrices of dimension $(n+1) \times \Theta(n^k)$) to a set of $\Theta(n^k)$-by-$\Theta(n^k)$
 symmetric matrices, and imposes an additional positive semidefiniteness constraint. Also, $\LS_+$, the operator defined in~\cite{LovaszS91a} that utilizes positive semidefiniteness, is equivalent to $\SA_+^1$. In general, $\SA_+^k$ dominates $\LS_+^k$ (i.e., $k$ iterative applications of $\LS_+$ --- see~\cite{AuT16a} for a proof).

\subsection{The Lasserre operator}

We now turn our attention to the $\Las$ operator due to Lasserre \\ 
\cite{Lasserre01a}. While $\Las$ can be applied to semialgebraic sets, we restrict our discussion to its applications to polytopes contained in $[0,1]^n$. Gouveia, Parrilo and Thomas provided in~\cite{GouveiaPT10a} an alternative description of the $\Las$ operator, where $P_I$ is described as the variety of an ideal intersected with the solutions to a system of polynomial inequalities. Our presentation of the operator is closer to that in~\cite{Laurent03a} than to Lasserre's original description.
Given $P := \set{x \in [0,1]^n : Ax \leq b}$, and an integer $k \in [n]$,

\begin{enumerate}
\item
Let $\widehat{\Las}^k(P)$ denote the set of matrices $Y \in \mathbb{S}_+^{\A^+_{k+1}}$ that satisfy all of the following conditions:
\begin{itemize}
\item[($\Las1$)]
$Y[\F, \F] = 1$;
\item[($\Las2$)]
For every $i \in [m]$, define the matrix $Y^i \in \mathbb{S}^{\A^+_{k}}$ where
\[
Y^i[S|_1, S'|_1] := b_i Y[S|_1, S'|_1] - \sum_{j=1}^n A[i,j] Y[(S \cup \set{j})|_1, (S' \cup \set{j})|_1],
\]
and impose $Y^i \succeq 0$.
\item[($\Las3$)]
For every $\a_1,\a_2, \b_1, \b_2 \in \A^+_k$ such that $\a_1 \cap \b_1 = \a_2 \cap \b_2$, $Y[\a_1, \b_1] = Y[\a_2, \b_2]$.
\end{itemize}
\item
Define
\[
\Las^k(P) := \set{x \in \mathbb{R}^n: \exists Y \in \widehat{\Las}^k(P):  \het{x}(Ye_{\F})= \begin{pmatrix} 1 \\ x \end{pmatrix}}.
\]
\end{enumerate}

For all operators $\Gamma$ considered in this paper, and for every polytope $P \subseteq [0,1]^n$, we define
$\Gamma^0(P):=P$.

We note that, unlike the previously mentioned operators, $\Las$ requires an explicit description of $P$ in terms of valid inequalities. While it is not apparent in the above definition of the $\Las$ operator (as it only uses the variables in the form $S|_1$, instead of the broader family of $S|_1 \cap T|_0$ as in operators based on $\SA$), we show that $\Las$ does commute with all automorphisms of the unit hypercube.

\begin{prop}\label{Lassym}
Let $L : \mR^n \to \mR^n$ be an affine transformation such that $\set{L(x) : x \in [0,1]^n} = [0,1]^n$. Then, $\Las^k(L(P)) = L(\Las^k(P))$ for all polytopes
$P \subseteq [0,1]$ and for every positive integer $k$.
\end{prop}

\begin{proof}
Since the automorphism group of the unit hypercube is generated by linear transformations swapping two coordinates and affine transformations flipping a coordinate, it suffices to prove that $\Las$ commutes
with each of these transformations.
First, we show that $\Las^k$ commutes with the mappings which swap two coordinates.
Without loss of generality, we may assume the coordinates are 1 and 2. Let $L_1$ denote the
linear transformation, where
\[
\left[L_1(x)\right]_i := \left\{
\begin{array}{ll}
x_2 & \tn{if $i=1$;}\\
x_1 & \tn{if $i=2$;}\\
x_i & \tn{otherwise.}
\end{array}
\right.
\]
We also define the map $\L : \A_{k+1}^+ \to \A_{k+1}^+$ where
\[
\L(S|_1) := \left\{
\begin{array}{ll}
((S \sm \set{1}) \cup \set{2} )|_1 & \tn{if $1 \in S, 2 \not\in S$;}\\
((S \sm \set{2}) \cup \set{1} )|_1 & \tn{if $2 \in S, 1 \not\in S$;}\\
S|_1 & \tn{otherwise,}
\end{array}
\right.
\]
Now suppose $x \in \Las^k(P)$, with certificate matrix $Y \in \widehat{\Las}^k(P)$. We show that $L_1(x) \in \Las^k(L_1(P))$. Define $Y' \in \mS^{\A_{k+1}^+}$ such that
\[
Y'[S|_1, T|_1] := Y[ \L(S)|_1 , \L(T)|_1 ], \,\,\,\, \textup{for all $S,T \in \A_{k+1}$.}
\]
 Then we see that $Y'$ is $Y$ with some columns and rows permuted, and thus is positive semidefinite too. Next, for each $a \in \mR^{n+1}$ such that $a_0 + \sum_{i=1}^n a_i x_i \geq 0$ is an inequality in the system describing $P$, define $a' \in \mR^{n+1}$ where
\[
a'_i :=  \left\{
\begin{array}{ll}
a_2 & \tn{if $i=1$;}\\
a_1 & \tn{if $i=2$;}\\
a_i & \tn{otherwise.}
\end{array}
\right.
\]
Then the collection of the derived inequalities $a'_0 + \sum_{i=1}^n a'_i x_i \geq 0$ describe $L(P)$. If this is the $j^{\tn{th}}$ inequality describing $L(P)$, then
\begin{eqnarray*}
&& Y'^{j}[S|_1,T|_1]\\
 &=& a'_0 Y'[S|_1, T|_1] + \sum_{i=1}^n a'_i Y'[(S \cup \set{i})|_1, (T \cup \set{i})|_1] \\
&=&  a_0 Y[ \L(S) |_1, \L(T)|_1] + a_2 Y[\L(S \cup \set{1})|_1, \L(T \cup \set{1})|_1]+ a_1 Y[\L(S \cup \set{2})|_1, \L(T \cup \set{2})|_1] \\
&&+\sum_{i=3}^n Y[\L(S \cup \set{i})|_1, \L(T \cup \set{i}))|_1] \\
&=&  a_0 Y[ \L(S) |_1, \L(T)|_1] + a_2 Y[\L(S) \cup \set{2}|_1, \L(T) \cup \set{2}|_1] + a_1 Y[\L(S) \cup \set{1}|_1, \L(T) \cup \set{1}|_1] \\
&&+ \sum_{i=3}^n Y[\L(S \cup \set{i})|_1, \L(T \cup \set{i}))|_1] \\
&=&  Y^j [ \L(S)|_1, \L(T)|_1].
\end{eqnarray*}
Thus, $Y'^j$ is also $Y^j$ with rows and columns permuted, and thus is positive semidefinite. Hence, we obtain that $\het{x}(Y' e_{\F}) = L_1(x)$ is in $\Las(L_1(P))$.

Next, consider the affine transformations flipping a coordinate (without loss of generality, the first coordinate).  So, we define $L_2 : \mR^n \to \mR^{n}$ where
\[
\left[L_2(x)\right]_i := \left\{
\begin{array}{ll}
1 - x_1 & \tn{if $i=1$;}\\
x_i & \tn{otherwise.}
\end{array}
\right.
\]
Also, for every integer $\l \geq 1$, define $U^{(\l)} \in \mR^{\A_{\l}^+ \times \A_{\l}}$ such that
\[
U^{(\l)}[S|_1, T|_1 \cap W|_0]:= \left\{
\begin{array}{ll}
 (-1)^{ |S \sm T|} & \tn{if $T \subseteq S \subseteq T \cup W$;}\\
0 & \tn{otherwise.}
\end{array}
\right.
\]
Now let $x \in \Las^k(P)$, with certificate matrix $Y \in \widehat{\Las}^k(P)$. Define $\bar{Y} \in \mS^{\A_{k+1}}$ where $\bar{Y} := (U^{(k+1)})^{\top} Y U^{(k+1)}$. This time, we let $\L : \A_{k+1}^+ \to \A_{k+1}$ denote the map where
\[
\L(S|_1) := \left\{
\begin{array}{ll}
((S \sm \set{1})|_1 \cap 1|_0 & \tn{if $1 \in S$;}\\
S|_1 & \tn{otherwise,}
\end{array}
\right.
\]
and let  $Y' \in \mS^{\A_{k+1}^+}$ such that
\[
Y'[S|_1, T|_1] := \bar{Y}[ \L(S)|_1 , \L(T)|_1 ], \,\,\,\, \tn{for all $S,T \in \A_{k+1}$.}
\]
Then we see that $Y'$ is a symmetric minor of $\bar{Y} = (U^{(k+1)})^{\top}YU^{(k+1)}$. Since $Y \succeq 0$, it follows that $Y' \succeq 0$ as well. Next, for each $a \in \mR^{n+1}$ such that $a_0 + \sum_{i=1}^n a_i x_i \geq 0$ is an inequality in the system describing $P$, define $a' \in \mR^{n+1}$ where
\[
a'_i :=  \left\{
\begin{array}{ll}
a_0 + a_1 & \tn{if $i=0$;}\\
-a_1 & \tn{if $i=1$;}\\
a_i & \tn{otherwise.}
\end{array}
\right.
\]
Then the collection of the derived inequalities $a'_0 + \sum_{i=1}^n a'_i x_i \geq 0$ describe $L(P)$. If this is the $j^{\tn{th}}$ inequality describing $L(P)$, then
\begin{eqnarray*}
&& Y'^{j}[S|_1,T|_1]\\
 &=& a'_0 Y'[S|_1, T|_1] + \sum_{i=1}^n a'_i Y'[(S \cup \set{i})|_1, (T \cup \set{i})|_1] \\
&=& (a_0 +a_1) \bar{Y}[ \L(S) |_1, \L(T)|_1] - a_1 \bar{Y}[\L(S \cup \set{1})|_1, \L(T \cup \set{1})|_1] \\
&&+  \sum_{i=2}^n \bar{Y}[\L(S \cup \set{i})|_1, \L(T \cup \set{i})|_1] \\
&=& a_0 \bar{Y}[ \L(S) |_1, \L(T)|_1] + a_1 \bar{Y}[\L(S) \cup \set{1}|_1, \L(T)\cup \set{1}|_1] \\
&&+  \sum_{i=2}^n \bar{Y}[\L(S) \cup \set{i}|_1, \L(T) \cup \set{i}|_1] \\
&=& ((U^{(k)})^{\top} Y^{j} U^{(k)})[ \L(S)|_1, \L(T)|_1].	
\end{eqnarray*}
Thus, $Y'^j$ is a symmetric minor of $(U^{(k)})^{\top} Y^{j} U^{(k)}$, and thus is positive semidefinite. Therefore, $\het{x}(Y' e_{\F}) = L_2(x)$ is in $\Las(L_2(P))$.
\end{proof}


\subsection{The Bienstock--Zuckerberg operator}

In~\cite{BienstockZ04a}, Bienstock and Zuckerberg devised a positive semidefinite lift-and-project operator (which we denote $\BZ_+$ herein) that is quite different from the previously (pre-2004) proposed operators. In particular, in its lifted space, it utilizes variables in $\A$ that are not necessarily in the form $S|_1 \cap T|_0$, in addition to a number of other ideas. One such idea is \emph{refinement}. While $\BZ_+$ is defined for any polytope contained in $[0,1]^n$, we will restrict our discussion to lower-comprehensive polytopes for simplicity's sake. Let polytope $P := \set{x \in [0,1]^n : Ax \leq b}$, where $A \in \mathbb{R}^{m \times n}$ is nonnegative and $b \in \mathbb{R}^m$ is positive (this implies that $P$ is lower-comprehensive; conversely, every $n$-dimensional lower-comprehensive polytope in $[0,1]^n$ admits such a representation). Given a vector $v$, let $\supp(v)$ denote the \emph{support} of $v$.

Next, a subset $O$ of $[n]$ is  called a \emph{$k$-small obstruction} of $P$ if there exists an inequality $a^{\top}x \leq b_i$ in the system $Ax \leq b$ where
\begin{itemize}
\item
$O \subseteq \supp(a)$;
\item
$\sum_{j \in O} a_j > b_i$; and
\item
$|O| \leq k+1$ or $|O| \geq |\supp(a)| - (k+1)$.
\end{itemize}

Observe that, given such an obstruction $O$,  the inequality $\sum_{i \in O} x_i \leq |O| - 1$ holds for every integral vector $x \in P$. Thus, if we let $\O_k$ denote the collection of all $k$-small obstructions of the system $Ax\leq b$, then the set
\[
\O_k(P) := \set{x \in P: \sum_{i \in O} x_i \leq |O|-1,~\forall O \in \O_k}
\]
is a relaxation of $P_I$ that is potentially tighter than $P$. The operator $\BZ_+$ then defines other collections of indices called walls and tiers, and uses these sets to construct the lifted space of $P$. In some rare cases though, when the system $Ax \leq b$ does not have a single $k$-small obstruction, we have the following result that relates the performance of $\SA_+$ and $\BZ_+'$ (a strengthened version of $\BZ_+$ defined in~\cite{AuT16a}):

\begin{prop}\label{noobs}
If $P = \set{ x \in [0,1]^n : Ax \leq b}$ where $Ax \leq b$ does not have a single $k$-small obstruction, then
\[
\SA_+^{2k}(P) \subseteq \BZ_+'^k(P).
\]
\end{prop}

\begin{proof}
If $P$ does not have a single obstruction, then $\O_k(P)= P$. Also, the collection of walls generated by $\BZ_+'^k$ consists of just the singleton sets. Thus, every tier (which is a union of up to $k$ walls) has size at most $k$. Then it is vacuously true that every tier of size greater than $k$ is $P$-useless (this concept of $P$-\emph{useless} is defined in \cite{AuT16a}), and thus by Proposition 4 in~\cite{AuT16a}, we obtain that $\SA_+'^k(P) \subseteq \BZ_+'^k(P)$. Since $\SA_+^{2k}(P) \subseteq \SA_+'^{k}(P)$ in general, our claim follows.
\end{proof}

The operator $\SA_+'$ mentioned the preceding proof is a strengthened version of $\SA_+$ (with additional constraints that are very similar to those differentiating $\BZ_+'$ from $\BZ_+$). To minimize notation and distraction, we have elected to only state elements of these operators that are crucial for the subsequent results we present. In Figure~\ref{fig0} we provide a comparison of relative strengths of all aforementioned lift-and-project operators, in addition to $\BCC$, a simple operator defined by Balas, Ceria, and Cornu{\'e}jols in~\cite{BalasCC93a}; and $\tilde{\LS}$, a geometric operator studied in~\cite{GoemansT01a} in their analysis of the Lov\'asz--Schrijver operators. Each arrow in the figure denotes ``is dominated by'', meaning that when applied to the same relaxation $P$,  the operator  at the head of an arrow would return a relaxation that is at least as tight as that obtained by applying the operator at the tail of the arrow. While the focus in this paper will be on the performance of $\SA_+, \Las$ and $\BZ_+'$, some of our results also have implications on these other operators. The reader may refer to~\cite{AuT16a} for the detailed definitions and some more intricate properties of these operators.

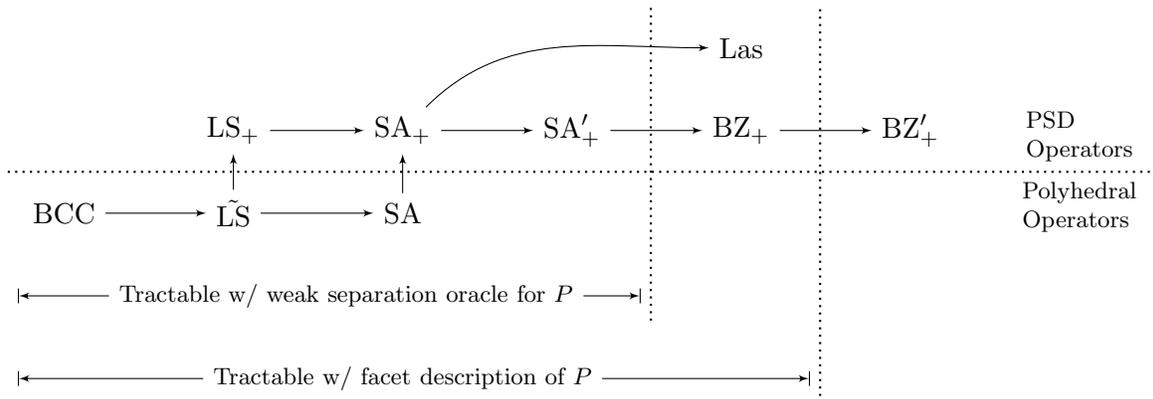
\begin{figure}[htb]
\begin{center}
\begin{tikzpicture}[y=1.1cm, x=1.5cm,>=latex', main node/.style={}, word node/.style={font=\footnotesize}]
\def\x{0}

\node at (1.5,0) (BCC) {$\BCC$};
\node at (3,0) (LStilde) {$\tilde{\LS}$};
\node at (4.5,0) (SA) {$\SA$};
\node at (3,1) (LS_+) {$\LS_+$};
\node[main node] at (4.5,1) (SA_+) {$\SA_+$};
\node[main node] at (6,1) (SA'_+) {$\SA'_+$};
\node at (7.5,1) (BZ_+) {$\BZ_+$};
\node[main node] at (9,1) (BZ'_+) {$\BZ_+'$};
\node at (7.5,2) (Las) {$\Las$};

        \draw[->] (BCC) -- (LStilde);
        \draw[->] (LStilde) -- (LS_+);
        \draw[->] (LStilde) -- (SA);

        \draw[->] (LS_+) -- (SA_+);
        \draw[->] (SA_+) -- (SA'_+);
        \draw[->] (SA'_+) -- (BZ_+);
        \draw[->] (BZ_+) -- (BZ'_+);

        \draw[->] (SA) -- (SA_+);

        \draw[->] (SA_+) to [out= 45, in = 180] (Las);

\def\y{-1} 

\draw[dotted,thick] (1,0.5) -- (11.2,0.5) ;
\draw[dotted,thick] (6.7,\y - 0.3) -- (6.7,2.5) ;
\draw[dotted,thick] (8.2,\y - 1.3) -- (8.2,2.5) ;

\node[word node, align=left, above] at (10.5,0.5)%
{PSD\\ Operators};
\node[word node, align=left, below] at (10.5,0.5)%
{Polyhedral\\ Operators};

\node[word node, align=left] at  (4,\y) (tract1)%
{Tractable w/ weak separation oracle for $P$};

\node[word node, align=left] at  (4.5, \y - 1) (tract2)%
{Tractable w/ facet description of $P$};

\node at (1,\y) (n1) {};
\node at (6.7,\y) (n2) {};
\node at (1,\y-1) (n3) {};
\node at (8.2,\y-1) (n4) {};

\path[->|]
(tract1) edge (n1)
(tract1) edge (n2)
(tract2) edge (n3)
(tract2) edge (n4);

\end{tikzpicture}
\end{center}
\caption{A strength chart of some lift-and-project operators.}\label{fig0}
\end{figure}

 There are also many other operators whose relative performance can be studied in this wider context of operators.  For example, recently Bodur, Dash and G\"{u}nl\"{u}k \cite{BodurDG2016} proposed a polyhedral lift-and-project operator called $\tilde{\textup{N}}$ and showed that
\[
\LS \rightarrow \tilde{\textup{N}} \rightarrow \SA^2
\]
where $\LS$ is a polyhedral operator devised in~\cite{LovaszS91a} that dominates $\tilde{\LS}$. 

Considering Figure~\ref{fig0}, note that every lower bound that we prove on rank as well as integrality gaps for $\Las$ and $\BZ'_+$ imply the same results for all other operators in Figure~\ref{fig0}.  Similarly, every upper bound on rank and integrality gaps for $\tilde{\LS}$ applies to all other operators in Figure~\ref{fig0}, except $\BCC$.

\section{Some bad instances for $\SA_+, \Las$ and $\BZ_+'$}

In this section, we consider several polytopes that have been shown to be bad instances for many known lift-and-project operators (and cutting plane schemes in general).

\subsection{The chipped hypercube $P_{n,\r}$}

Recall the chipped hypercube
\[
P_{n,\r} := \set{ x \in [0,1]^n : \sum_{i=1}^n x_i  \leq n- \r}.
\]
Cook and Dash~\cite{CookD01a} showed that the $\LS_+$-rank of $P_{n,1/2}$ is $n$, while Laurent~\cite{Laurent03a} proved that the $\SA$-rank of $P_{n,1/2}$ is also $n$. Cheung~\cite{Cheung07a} extended these results and showed that both the $\LS_+$- and $\SA$-rank of $P_{n,\r}$ are $n$ for all $\r \in (0,1)$. Here, we use similar techniques to establish the $\SA_+$-rank for $P_{n,\r}$. Note that, from here on, we will sometimes use $v[i]$ to denote the $i$-entry of a vector $v$ (instead of $v_i$).

\begin{prop}\label{Kndstrong}
For every $n \geq 2$, the $\SA_+$-rank of $P_{n,\r}$ is $n$ for all $\r \in (0,1)$.
\end{prop}

\begin{proof}
We prove our claim by showing that $\bar{x} := \left(1 - \frac{\r}{n\r + 1 - \r} \right)\bar{e} \in \SA_+^{n-1}\left(P_{n,\r}\right) \sm \left( P_{n,\r} \right)_I$, where $\bar{e}$ denotes the all-ones vector. First,
\[
\sum_{i=1}^{n} \bar{x}_i = n\left(1 - \frac{\r}{n\r + 1 - \r} \right) = n - \frac{n\r}{n\r + 1-\r} > n- 1,
\]
and so $\bar{x} \not\in \left( P_{n,\r} \right)_I$. We next show that this vector is in $\SA_+^{n-1}\left(P_{n,\r}\right)$. Define $Y \in \mR^{\A_{n-1} \times \A_{n-1}}$ such that
\[
Y[\a, \b] := \left\{
\begin{array}{ll}
1 - \frac{\r |S|}{n\r + 1 - \r} & \tn{if $\a \cap \b = S|_1$ for some $S \subseteq [n]$;}\\
\frac{\r}{n\r + 1 - \r} & \tn{if $\a \cap \b = S|_1 \cap j|_0$ for some $S \subseteq [n]$ and $j \in [n] \sm S$;}\\
0 & \tn{otherwise.}
\end{array}
\right.
\]
We claim that $Y \in \widehat{\SA}_+^{n-1}\left(P_{n,\r}\right)$. First, ($\SA_+ 1$) holds as $Y[\F,\F] = 1$ (since $\F \cap \F = \es|_1$). It is not hard to see that $Y \geq 0$, as every entry in $Y$ is either $0, \frac{\r}{n\r + 1 - \r}$ or $1 -  \frac{k\r}{n\r + 1 - \r}$ for some integer $k \in \set{0,\ldots, n}$. Next, we check that $\het{x}(Y e_{\b}) \in K\left(P_{n,\r}\right)$ for all $\b \in \A_{n-1}$. Given $\b = S|_1 \cap T|_0, \het{x}(Y e_{\b})$ is the zero vector whenever $|T| \geq 2$, and is the vector $\frac{\r}{n\r + 1 - \r} (\bar{e} - e_i)$ whenever $T = \set{i}$ for some $i \in [n]$.

Finally, suppose $\b = S|_1$ for some $S \subseteq [n]$ where $|S| = k$. 
Then
\[
 \het{x}(Ye_{\b})[i] = \left\{
\begin{array}{ll}
1 - \frac{k\r}{n\r + 1-\r} & \tn{if $i = 0$ or $i \in S$;}\\
1 - \frac{(k+1)\r}{n\r + 1-\r} & \tn{if $i \in [n] \sm S$.}
\end{array}
\right.
\]
Now
\begin{eqnarray*}
\sum_{i=1}^n \het{x}(Ye_{\b})[i] &=& k \left( 1 - \frac{k\r}{n\r + 1-\r}\right) +  (n-k) \left( 1 - \frac{(k+1)\r}{n\r + 1-\r}\right)\\
&=& n\left( 1 - \frac{k\r}{n\r + 1-\r}\right) - \r\left( \frac{n-k}{n\r + 1-\r}\right) \\
& \leq & (n-\r)\left( 1 - \frac{k\r}{n\r + 1-\r}\right)\\
&=& (n-\r) \het{x}(Ye_{\b})[0].
\end{eqnarray*}
Thus, $\het{x}(Ye_{\b}) \in K(P)$ in this case as well. Next, it is not hard to see that the entries of $Y$ satisfy ($\SA_+ 3$), ($\SA_+4$) and ($\SA_+ 5$). Finally, to see that $Y \succeq 0$, let $Y'$ be the symmetric minor of $Y$ indexed by rows and columns from $\A_{n-1}^- := \set{S|_0 : S \subseteq [n], |S| \leq n-1}$. Then $Y' \succeq 0$ as it is diagonally dominant. Next, define $L \in \mR^{\A_{n-1} \times \A_{n-1}^-}$ where
\[
L[S|_1 \cap T|_0, U|_0] := \left\{
\begin{array}{ll}
(-1)^{|S|}& \tn{if $S \cup T = U$;}\\
0 & \tn{otherwise.}
\end{array}
\right.
\]
Then it can be checked that $Y = LY'L^{\top}$. Hence, we conclude that $Y \succeq 0$ as well. This completes our proof.
\end{proof}

We next show that $(0,1)$ is the only range of $\r$'s for which the $\SA_+$-rank of $P_{n,\r}$ is $n$. To do that, it is helpful to introduce the notion of \emph{moment matrices}. Given an integer $k \geq 0$ and vector $y \in \mR^{\A_{\l}^+}$ where $\l \geq \min\set{n, 2k}$, we define the matrix $\M_{k}(y) \in \mR^{\A_{k}^+ \times \A_{k}^+}$ where $\M_{k}(y)[\a,\b] := y[\a \cap \b]$ for all $\a,\b \in \A_{k}^+$. Then we have the following:

\begin{prop}\label{SA+PalphaUB}
For every $n \geq 2$ and non-integer $\r \in (0,n)$, the $\SA_+$-rank of $P_{n,\r}$ is at most $n- \lceil \r \rceil + 1$.
\end{prop}

\begin{proof}
Let $P:=P_{n,\r}$. We use the notion of $\l$-establishment (defined in ~\cite{AuT16a}) to prove this claim. First, let $\l := n - \lceil \r \rceil$. From the conditions imposed by $\SA_+$, it is easy to check that every matrix $Y \in \widehat{\SA}_+^{\l+1}(P)$ is $(\l+1)$-established (and thus $\l$-established). Also, the condition $(\SA_+ 5)$ guarantees that the symmetric minor of $Y$ with rows and columns indexed by sets in $\A_{\l}^+$ is  a moment matrix $\M_{\l}(y)$ for some vector $y$.

Now notice that for every set $S \subseteq [n]$ where $|S| = \l+1$, the incidence vector of $S$ is not in $P$ (as $\l+1 > n-\r$). Hence, the condition $\het{x}\left( Y e_{S|_1}\right) \in K(P)$ imposed by $(\SA_+ 2)$ implies that $Y[\F, S|_1] = 0$. Thus, we obtain that $Y[S|_1, S|_1] = 0$ for all $S \subseteq [n]$ of size $\l+1$, and so the diagonal entries of  the symmetric minor $Y'$ of $Y$ indexed by sets in $\A_{\l+1}^+ \sm \A_{\l}$ are all zero. For $Y$ to be positive semidefinite, $Y'$ must have all zero entries. Thus, we obtain that $y[S|_1] = 0$ for all sets $S$ where $|S| \geq \l+1$.

Next, if we define $Z_i := \sum_{S \subseteq [n], |S|=i} y[S|_1]$ for every $i \geq 0$, we obtain that $Z_{i} = 0$ for all $i > \l$. Then it follows from Corollary 12 in~\cite{AuT16a} that $Z_1 \leq \l$. Since
\[
Z_1 = \sum_{i=1}^n y[i|_1] = \sum_{i=1}^n Y[i|_1, \F],
\]
we conclude that $\sum_{i=1}^n x_i \leq \l$ is valid for $\SA_+^{\l+1}(P)$, and our claim follows.
\end{proof}

Thus, we know that the $\SA_+$-rank of $P_{n,\r}$ is exactly $n$ when $\r \in (0,1)$, and the rank is $1$ if $\r \in (n-1,n)$. When $\r \in (n-2,n-1)$, it follows from Proposition~\ref{SA+PalphaUB} that the $\SA_+$-rank is at most $2$. Since it is not hard to show that $\SA_+^1 \left( P_{n,\r} \right) \neq P_{n,n-1}$, we know in this case that the $\SA_+$-rank is exactly $2$.

Next, we show that for a weaker operator, the rank of $P_{n,\r}$ is always $n$ if it is not integral and strictly contains the unit simplex. Given integer $k \in [n]$ and $P \subseteq [0,1]^n$, consider the following operator originally defined in \cite{GoemansT01a}:
\[
\tilde{\LS}^k(P) := \bigcap_{S \subseteq [n], |S|=k} \conv \set{ x \in P: x_i \in \set{0,1}, ~\forall i \in S}.
\]
That is, $x$ is in $\tilde{\LS}^k(P)$ if and only if for every set of indices $S$ of size $k$, $x$ can be expressed as a convex combination of points in $P$ whose entries in $S$ are all integral. While $\tilde{\LS}$ produces tighter relaxations than  $\BCC$, it in turn is dominated by $\SA$ and several operators devised by Lov\'asz and Schrijver in~\cite{LovaszS91a} (see, for instance,~\cite{GoemansT01a} for a discussion on this matter). Then we have the following:

\begin{prop}\label{GammaLB}
For every integer $n \geq 2$ and for every non-integer $\r \in (0,n-1)$, the $\tilde{\LS}$-rank of $P_{n,\r}$ is $n$.
\end{prop}

\begin{proof}
Let $P := P_{n,\r}$ and $\l := n - \lceil \r \rceil$ (so $P_I = P_{n,\l}$). We prove our claim by showing that $\max\set{ \bar{e}^{\top} x : x \in \tilde{\LS}^{n-1}(P)} > \l$.

First, let $S = [n-1]$, and define $\epsilon := \min \set{ \lceil \r \rceil - \r, \frac{\l}{n-1}}$. Also, given $T \subseteq S$, let $\chi_T$ denote the incidence vector of $T$ in $\set{0,1}^{n-1}$. Now consider the point
\[
\bar{x} := \left(\sum_{T \subseteq S, |T| = \l} \frac{n-\l}{\binom{n-1}{\l}(n - \epsilon(n-1))} \begin{pmatrix} \chi_T \\ \epsilon \end{pmatrix}
 \right) +\left(\sum_{T \subseteq S, |T| = \l-1} \frac{\l - \epsilon(n-1)}{\binom{n-1}{\l-1}(n - \epsilon(n-1))} \begin{pmatrix} \chi_T \\ 1 \end{pmatrix}
 \right).
\]
First, observe that $\bar{x}$ is a linear combination of the points whose entries in $S$ are integral. Also, $\begin{pmatrix} \chi_T \\ \epsilon \end{pmatrix} \in P$ for all $T$ of size $\l$ (by the choice of $\epsilon$), and $\begin{pmatrix} \chi_T \\ 1 \end{pmatrix} \in P$ for all $T$ of size $\l-1$ as well.

Furthermore, since $\epsilon(n-1) \leq \l$, the weights on these points are nonnegative, and do sum up to $1$. Thus, $\bar{x}$ is indeed a convex combination of these points. By the symmetry of $P$ and the definition of $\tilde{\LS}$, we can express $\bar{x}$ as a similar convex combination of points in $P$ for all other sets $S$ of size $n-1$. Thus, this shows that $\bar{x} \in \tilde{\LS}^{n-1}(P)$.

On the other hand, it is easy to check that $\bar{x} = \frac{\l (1-\epsilon) + \epsilon}{n(1-\epsilon) + \epsilon} \bar{e}$, and thus $\bar{e}^{\top} \bar{x} > \l$ and $\bar{x} \not\in P_I$. Hence, we deduce that $P$ has $\tilde{\LS}$-rank $n$.
\end{proof}

Thus, we see that when $\r$ is close to $n-1$, the positive semidefiniteness constraint imposed by $\SA_+$ is in fact helpful in generating the desired facet of the integer hull that can be elusive to a weaker polyhedral operator until the $n^{\tn{th}}$ iteration.

We next give a lower bound on the $\SA_+$-rank of $P_{n,\r}$ for some cases where $\r > 1$, which will be useful when we later establish a $\BZ_+'$-rank lower bound for some of these polytopes. We first need the following result. Suppose $P \subseteq [0,1]^n$. Given $x \in P$, let
\[
S(x) := \set{ i \in [n] : 0 < x_i < 1}.
\]
Also, given $x \in [0,1]^n$ and two disjoint sets of indices $I,J \subseteq [n]$, we define the vector $x^I_J \in [0,1]^n$ where
\[
x^I_J[i] := \left\{
\begin{array}{ll}
1 & \tn{if $i \in I$;}\\
0 & \tn{if $i \in J$;}\\
x[i] & \tn{otherwise.}
\end{array}
\right.
\]
In other words, $x^I_J$ is the vector obtained from $x$ by setting all entries indexed by elements in $I$ to 1, and all entries indexed by elements in $J$ to 0. Then we have the following useful property that is inherited by a wide class of lift-and-project operators.

\begin{lem}[Theorem 15 in~\cite{AuT16a}]\label{SA_+}
Let $P \subseteq [0,1]^n$ and $x \in P$. If $x^I_J \in P$ for all $I,J \subseteq S(x)$ such that $|I| + |J| \leq k$, then $x \in \SA_{+}^k(P)$.
\end{lem}

Using Lemma~\ref{SA_+}, we have the following for the $\SA_+$-rank of $P_{n,\r}$:

\begin{prop}\label{Knd}
For every $n \geq 2$, if $\r \in (0,n)$ is not an integer and $k < \frac{n(\lceil \r \rceil - \r)}{ \lceil \r \rceil}$, then the $\SA_+$-rank of $P_{n,\r}$ is at least $k+1$.
\end{prop}

\begin{proof}
First, observe that
\[
k < \frac{n(\lceil \r \rceil - \r)}{ \lceil \r \rceil} \iff (n-k) \left(\frac{ n- \lceil \r \rceil}{n}\right) + k < n - \r.
\]
Thus, there exists $\l \in \mR$ such that $ (n-k) \l + k < n- \r$ and $\l > \frac{n- \lceil \r \rceil}{n}$. Consider the point $\bar{x} := \l \bar{e}$. Since $\l > \frac{n- \lceil \r \rceil}{n}$, $\bar{x} \not \in \left( P_{n,\r} \right)_I$. However, for every pair of 
disjoint sets of indices $I,J \subseteq [n]$ where $|I| + |J| \leq k$, we have
\[
\sum_{i=1}^{n} \bar{x}^I_J[i] \leq (n-k) \l + k < n- \r,
\]
by the choice of $\l$. Thus, $\bar{x}^I_J \in P_{n,\r}$ for all such choices of $I,J$. (Note that the first inequality above follows from the fact that $\sum_{i=1}^n \bar{x}^I_J[i]$ is maximized by choosing $I,J$ where $|I| =k$ and $J = \es$.) Thus, it follows from Lemma~\ref{SA_+} that $\bar{x} \in \SA_+^k\left(P_{n,\r}\right)$. This proves that $\SA_+^k\left(P_{n,\r}\right) \neq \left( P_{n,\r} \right)_I$, and hence the $\SA_+$-rank of $P_{n,\r}$ is at least $k+1$.
\end{proof}

Using Proposition~\ref{Knd}, we obtain a lower-bound result on the $\BZ_+'$-rank of $P_{n,\r}$, establishing what we believe to be the first example in which $\BZ_+'$ (and, as a result, $\BZ_+$) requires more than a constant number of iterations to return the integer hull of a set.

\begin{thm}\label{BZ_+'Palpha}
Suppose an integer $n\geq 5$ is not a perfect square. Then there exists $\r \in (\left\lfloor \sqrt{n} \right\rfloor, \lceil \sqrt{n} \rceil)$ such that the $\BZ_+'$-rank of $P_{n,\r}$ is at least $\left\lfloor \frac{ \sqrt{n}+1}{2} \right\rfloor$.
\end{thm}

\begin{proof}
Let $P := P_{n,\r}$. First, choose $\epsilon \in (0,1)$ small enough such that
\[
\sqrt{n} - 1 < \frac{ n \left(1 - \epsilon \right)}{ \lceil \sqrt{n} \rceil},
\]
and let $\r := \left\lfloor \sqrt{n} \right\rfloor + \epsilon$. Next, let $k := \left\lfloor\frac{ \sqrt{n}-1}{2} \right\rfloor$. Notice that for all $n \geq 5$, $k+1 <  \r < n-(k+1)$, and so $\BZ_+'^k$ does not generate any $k$-small obstructions for $P$. Thus, we obtain that $\SA_+^{2k}(P) \subseteq \BZ_+'^k(P)$ by Proposition~\ref{noobs}. Also, from Proposition~\ref{Knd}, since $2k \leq \sqrt{n} - 1$, $\SA_+^{2k}(P)  \neq P_I$. Thus, the $\BZ_+'$-rank of $P$ is at least $k+1 = \left\lfloor \frac{ \sqrt{n}+1}{2} \right\rfloor$.
\end{proof}

We note that the $\BZ_+$-rank of $P_{n,\r}$ is 1 for every $\r \in (0,1)$. This is because the set $[n]$ is a $k$-small obstruction for every $k \geq 1$, and so $\sum_{i=1}^n x_i \leq n-1$ is valid for $\O_k\left(P_{n,\r}\right)$, and the refinement step in $\BZ_+$ already suffices in generating the integer hull of $P_{n,\r}$. More generally, when  $k +1 \geq \r$, every subset of set of $[n]$ of size $n-k$ does qualify as a $k$-small obstruction, and it can be shown that $\BZ_+^k \left( P_{n,\r} \right) =\left( P_{n,\r} \right)_I$. On the other hand, since $\BZ_+$ (and the refined version $\BZ_+'$) dominates $\SA_+$, Proposition~\ref{SA+PalphaUB} implies that the $\BZ_+$-rank of $P_{n,\r}$ is at most $n- \lceil \r \rceil +1$. This implies that, in contrast with other operators (including $\SA_+$ and, as we will see, $\Las$), the $\BZ_+$-rank of $P_{n,\r}$ is low both when $\r$ is close to $0$ or $n$.

We next turn to the $\Las$-rank of $P_{n,\r}$. Interestingly, Cheung showed the following in~\cite{Cheung07a}:

\begin{thm}\label{CheungPalpha}
\begin{enumerate}[(i)]
\item
For every even integer $n \geq 4$, the $\Las$-rank of $P_{n,\r}$ is at most $n-1$ for all $\r \geq \frac{1}{n}$;
\item
For every integer $n \geq 2$, there exists $\r \in \left(0,\frac{1}{n}\right)$ such that the $\Las$-rank of $P_{n,\r}$ is $n$.
\end{enumerate}
\end{thm}

Thus, while the rank of $P_{n,\r}$ is invariant under the choice of $\r \in (0,1)$ with respect to all other lift-and-project operators we have considered so far, it is not the case for $\Las$. Next, we strengthen part (ii) of Cheung's result above, and give a range of $\r$ where $P_{n,\r}$ has $\Las$-rank $n$ for every $n \geq 2$.

\begin{thm}\label{LasPalpha}
Suppose $n \geq 2$, and
\[
0 <  \r \leq \frac{n^2-1}{2 n^{n+1} - n^2 -1}.
\]
Then $P_{n,\r}$ has $\Las$-rank $n$.
\end{thm}

Before we prove Theorem~\ref{LasPalpha}, we need some notation and lemmas. Define the matrix $Z \in \mathbb{R}^{\A_n^+ \times \A_n^+}$ where
\[
Z[S|_1, T|_1] := \left\{
\begin{array}{ll}
1 & \tn{if $S \subseteq T$;}\\
0 & \tn{otherwise.}
\end{array}
\right.
\]
$Z$ is the \emph{zeta matrix} of $[n]$. Note that $Z$ is invertible, and it is well known that its inverse is the \emph{M\"obius matrix} $M \in \mathbb{R}^{\A_n^+ \times \A_n^+}$ where
\[
M[S|_1, T|_1] := \left\{
\begin{array}{ll}
(-1)^{|T \sm S|} & \tn{if $S \subseteq T$;}\\
0 & \tn{otherwise.}
\end{array}
\right.
\]
Throughout this paper, we will assume that the rows and columns in $Z$ and $M$ are ordered such that the last row/column corresponds to the set $[n]|_1$. Note that, with such an ordering, the last column of $Z$ is the all-ones vector. The following relation between zeta matrices and moment matrices is due to Laurent~\cite{Laurent03a}:

\begin{lem}[Lemma 2 in~\cite{Laurent03a}]\label{zetau}
Suppose $y \in \mR^{\A_n^+}$. Define $u \in \mR^{\A_n^+}$ where
\[
u[S|_1] := \sum_{T \supseteq S} (-1)^{|T \sm S|} y[T|_1].
\]
Then $\M_n(y) = Z \Diag(u) Z^{\top}$.
\end{lem}

Note that we used $\Diag(u)$ to denote the diagonal matrix $U$ where $U[S|_1,S|_1] := u[S|_1]$ for all $S \subseteq [n]$. Next, the following lemma will be useful for proving Theorem~\ref{LasPalpha}, as well as analyzing the cropped hypercube $Q_{n,\r}$ later on. Note that it uses very similar ideas to that in~\cite{KurpiszLM15a}, where they characterized general conditions for when $\M_{n-1}(w)$ is positive semidefinite, although the proof here is simpler as we are specifically focused on the applications to the sets $P_{n,\r}$ and $Q_{n,\r}$.

\begin{lem}\label{thetapower}
Let $\t \in (0,1)$ be a fixed number. Define $y \in \mR^{\A_{n}^+}$ such that $y[S|_1] := \t^{|S|}$ for all $S \subseteq [n]$. Then
\begin{enumerate}[(i)]
\item
$\M_i(y) \succeq 0$ for all $i \in [n]$.
\item
Given any $\r > 0$,
\[
(n-\r) \M_n(y) [S|_1,T|_1] - \sum_{i=1}^n \M_n(y) [(S \cup
\set{j})|_1, (T \cup \set{j})|_1] = \M_n(w)[S|_1,T|_1]
\]
where $w|S|_1] := \left((n-|S|)(1-\t) - \r \right) \t^{|S|}$ for all $S
\subseteq [n]$. Moreover,
\[
\M_n(w) = Z\Diag(u) Z^{\top},
\]
where $u[S|_1] := (n- |S| - \r) \t^{|S|}(1-\t)^{n-|S|}$ for all $S \subseteq [n]$.
\item
If $\r \in (0,1)$ and
\begin{equation}\label{alphabound}
\r \leq \frac{ (n+1) \t (1-\t)^n}{2 - \left[(n-1)\t+2\right](1-\t)^n},
\end{equation}
then $\M_{n-1}(w) \succ 0$.
\end{enumerate}
\end{lem}

\begin{proof}
To prove part (i), it suffices to show that $\M_n(y) \succeq 0$, as $\M_i(y)$ is a symmetric minor of $\M_n(y)$ for all $i < n$. By Lemma~\ref{zetau},  Since $\M_n(y) = Z \Diag(v) Z^{\top}$, where
\[
v[S|_1] = \sum_{T \supseteq S}  (-1)^{|T \sm S|} y[T|_1]  = \sum_{i= 0}^{n-|S|} \binom{n-|S|}{i} (-1)^i \t^{|S|+i} = \t^{|S|}(1-\t)^{n-|S|},
\]
which is positive for all $S \subseteq [n]$. Thus, it follows that $\M_n(y) \succeq 0$. For part (ii), we see that
\begin{eqnarray*}
&& (n-\r) \M_n(y) [S|_1,T|_1] - \sum_{i=1}^n \M_n(y) [(S \cup
\set{j})|_1, (T \cup \set{j})|_1] \\
&=& (n-\r) \t^{|S \cup T|} - \left( |S \cup T| \t^{|S \cup T|} + (n - |S \cup T|) \t^{|S \cup T| +1} \right)\\
&=& \left( (n- |S \cup T|)(1-\t) - \r \right) \t^{|S \cup T|} \\
&=& \M_n(w)[S|_1,T|_1].
\end{eqnarray*}
Also, it is not hard to check that $\sum_{T \supseteq S} (-1)^{|T \sm S|} w[S|_1] = u[S|_1]$ for all $S \subseteq [n]$, and so the last part of the claim follows from Lemma~\ref{zetau}.

Finally, for (iii), let $\bar{Z}$ and $\bar{M}$, respectively, denote the symmetric minor of $Z$ and $M$ with the row and column corresponding to $[n]|_1$ removed. We also let $u' \in \A_+^{n-1}$ denote the vector obtained from $u$ by removing the entry corresponding to $[n]|_1$. Then by Lemma~\ref{zetau},
\begin{eqnarray*}
\M_{n}(w)= Z \Diag(u) Z^{\top} &=& \begin{pmatrix} \bar{Z} & \bar{e} \\ 0 & 1 \end{pmatrix} \begin{pmatrix} \Diag(u') & 0 \\ 0 & -\t^{n}\r \end{pmatrix}\begin{pmatrix} \bar{Z}^{\top} & 0 \\ \bar{e}^{\top} & 1 \end{pmatrix} \\
&=& \begin{pmatrix} \bar{Z} \Diag(u') \bar{Z}^{\top} -\t^{n}\r \bar{e}\bar{e}^{\top} & -\t^{n}\r \bar{e} \\ -\t^{n}\r \bar{e}^{\top} & -\t^{n}\r \end{pmatrix}.
\end{eqnarray*}
Since $\M_{n-1}(w)$ is the symmetric minor of $\M_{n}(w)$ with the last row and column removed, we obtain that $\M_{n-1}(w) = \bar{Z} \Diag(u') \bar{Z}^{\top} - \t^{n} \r \bar{e}\bar{e}^{\top}$. Notice that $u[S|_1] =(n - |S| - \r) \t^{|S|}(1-\t)^{n-|S|} >0$ for all $S \subset [n]$, and that $\bar{M}$ is nonsingular ($\bar{M}$ is the inverse of $\bar{Z}$). Hence, $\left[\Diag(u')\right]^{-1/2}\bar{M}$ is nonsingular and
$\left[\Diag(u')\right]^{-1/2}\bar{M} \cdot \bar{M}^{\top}\left[\Diag(u')\right]^{-1/2}$ is an automorphism of the underlying cone of positive semidefinite matrices.
Therefore, $\M_{n-1}(w) \succ 0$ if and only if 
\[
Y := \left[\Diag(u')\right]^{-1/2}\bar{M} \M_{n-1}(w) \bar{M}^{\top}\left[\Diag(u')\right]^{-1/2}
\]
is positive definite. Now observe that $Y = I -\rho \t^n \xi \xi^{\top}$,
where $\xi := \left[\Diag(u')\right]^{-1/2}\bar{M}\bar{e}$.  Hence,
\begin{equation}
\label{eqn:momentpd}
\M_{n-1}(w) \succ 0 \iff \r\t^n \xi^{\top}\xi < 1.
\end{equation}
Next, using the fact that $(\bar{M}\bar{e})[S|_1] = (-1)^{n-|S|-1}$ for all $S \subset [n]$, we analyze $\r \t^n \xi^{\top}\xi$ which is equal to:
\begin{eqnarray*}
\r \t^{n} \left( \sum_{S \subset [n]} \frac{1}{u[S|_1} \right)  &=&  \r \t^{n} \left( \sum_{S \subset [n]} \frac{1}{(n-|S|- \r) \t^{|S|}(1-\t)^{n-|S|}} \right) \\
 &=&  \frac{\r \t^n}{(1-\t)^n}  \left( \sum_{i=0}^{n-1}\frac{1}{n-i-\r}\binom{n}{i} \left( \frac{1-\t}{\t} \right)^{i} \right) \\
 &< &  \frac{\r \t^n}{(1-\t)^n}  \left( \sum_{i=0}^{n-1}\frac{2}{(n+1)(1-\r)}\binom{n+1}{i} \left( \frac{1-\t}{\t} \right)^{i} \right) \\
 &= &  \frac{2 \r \t^n }{(1-\r)(n+1)(1-\t)^n} \left( \left(\frac{1}{\t} \right)^{n+1}  - (n+1) \left( \frac{1-\t}{\t} \right)^n - \left( \frac{1-\t}{\t} \right)^{n+1} \right) \\
&=& \frac{\r}{1-\r} \left( \frac{2 \left[1 - (n\t+1)(1-\t)^n \right]}{(n+1) \t (1-\t)^n}\right).
\end{eqnarray*}
Thus, if $\r \leq \frac{ (n+1) \t (1-\t)^n}{2 -\left[(n-1)\t+2\right](1-\t)^n }$, then $\M_{n-1}(w)$ is positive definite.
\end{proof}

We are now ready to prove Theorem~\ref{LasPalpha}.

\begin{proof}[Proof of Theorem~\ref{LasPalpha}]
It is obvious that $\left( P_{n,\r} \right)_I = P_{n,1}$ for all $\r \in (0,1)$. Now suppose we are given integer $n \geq 2$ and $0 <  \r \leq \frac{n^2-1}{2 n^{n+1} - n^2 -1}$. We prove our claim by showing that  there exists $\t > \frac{n-1}{n}$ where $\t \bar{e} \in \Las^{n-1}\left(P_{n,\r}\right)$.

Define $y \in \mR^{\A_n^+}$ where $y[S|_1] := \t^{|S|}$ for all $S \subseteq [n]$, then Lemma~\ref{thetapower} implies $Y := \M_n(y) \succeq 0$. It also implies that $Y^1= \M_{n-1}(w)$ where $w|S|_1] = ((n-|S|)(1-\t) -\r) \t^k$.
To prove our claim, it suffices to show that there exists $\t > \frac{n-1}{n}$ such that $\M_{n-1}(w) \succeq 0$. Since our upper bound on $\r$ is continuous in $\t$ in a neighbourhood of $\t=\frac{n-1}{n}$, and the
cone of positive definite matrices is the interior of the cone of positive semidefinite matrices, by \eqref{eqn:momentpd}, it suffices to show that $\M_{n-1}(w) \succ 0$ when $\t = \frac{n-1}{n}$. Then by letting $\t = \frac{n-1}{n}$ in~\eqref{alphabound} and simplifying, we obtain that $\r \leq \frac{n^2-1}{2 n^{n+1} - n^2 - 1}$ guarantees $\M_{n-1}(w) \succ 0$, and the claim follows.
\end{proof}

Let $p(n)$ denote the largest $\r >0$ where $\M_{n-1}(y) \in \widehat{\Las}^{n-1}\left(P_{n,\r}\right)$ for some $\theta > \frac{n-1}{n}$ (where $y$ is defined in the proof of Theorem~\ref{LasPalpha}). Figure~\ref{figPalpha} shows the value of $\log_n(p(n))$ for some small values of $n$, as well as the lower bound on $p(n)$ given by Theorem~\ref{LasPalpha}.

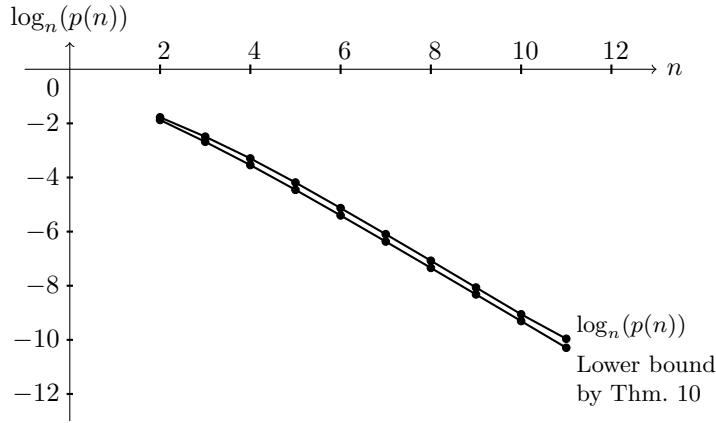
\begin{figure}[htb]
\begin{center}

\begin{tikzpicture}[scale = 0.6, xscale=1,yscale=0.6, font=\small,  word node/.style={font=\footnotesize}]]

\def\xlb{0};
\def\xub{12};
\def\ylb{-12};
\def\yub{0};
\def\buf{1};

\draw [->] (\xlb - \buf,0) -- (\xub + \buf,0);
\draw [->] (0, \ylb - \buf) -- (0, \yub + \buf);

\foreach \x in {\xlb ,2, ...,\xub}
{
    \ifthenelse{\NOT 0 = \x}{\draw[thick](\x ,-5pt) -- (\x ,5pt);}{}
    \ifthenelse{\NOT 0 = \x}{\node[anchor=south] at (\x,0) (label) {{ $\x$}};}{}
}

\foreach \y in {\ylb ,-10, ...,\yub}
{
   \ifthenelse{\NOT 0 = \y}{\draw[thick](-2pt, \y) -- (2pt, \y);}{}
 \ifthenelse{\NOT 0 = \y}{\node[anchor=east] at (0,\y) (label) {{ $\y$}};}{}

}

\node[anchor=north east] at (0,0) (label3) { $0$};

\node[anchor=west] at (\xub + \buf,0) (label3) {$n$};
\node[anchor=south] at (0, \yub+\buf) (label3) {$\log_n(p(n))$};

%

\draw[thick]
(    2.0000 ,  -1.8745)--(
    3.0000  , -2.6801)--(
    4.0000  , -3.5405)--(
    5.0000  , -4.4555)--(
    6.0000  , -5.4025)--(
    7.0000  , -6.3668)--(
    8.0000  , -7.3409)--(
    9.0000  , -8.3211)--(
  10.0000  , -9.3054)--(
   11.0000 , -10.2925);

\ignore{
   2.0000    0.5454
    3.0000    0.4737
    4.0000    0.4727
    5.0000    0.4804
    6.0000    0.4862
    7.0000    0.4898
    8.0000    0.4922
    9.0000    0.4938
   10.0000    0.4950
   11.0000    0.4959

 2.0000    0.5844
    3.0000    0.5837
    4.0000    0.6643
    5.0000    0.7425
    6.0000    0.7922
    7.0000    0.8312
    8.0000    0.8556
    9.0000    0.8669
   10.0000    0.8810
   11.0000    1.1007

}

\draw[thick]
(2, -1.775)--(3, -2.49)--(4, -3.295)--(5, -4.185)--(6, -5.13)--(7, -6.095)--(8, -7.075)--(9, -8.065)--(10, -9.055)--(11,-9.96);

\def\y{0.7};

\foreach \position in {(    2.0000 ,  -1.8745),(
    3.0000  , -2.6801),(
    4.0000  , -3.5405),(
    5.0000  , -4.4555),(
    6.0000  , -5.4025),(
    7.0000  , -6.3668),(
    8.0000  , -7.3409),(
    9.0000  , -8.3211),(
  10.0000  , -9.3054),(
   11.0000 , -10.2925),
(2, -1.775),
(3, -2.49),
(4, -3.295),
(5, -4.185),
(6, -5.13),
(7, -6.095),
(8, -7.075),
(9, -8.065),
(10, -9.055),
(11,-9.96)}
{
\node[draw, circle, inner sep=0pt, minimum size = 0.1cm, fill] at \position {};
}

\node[word node, align=left, right] at (   11.0000,  -9.5) {$\log_n(p(n))$};
\node[word node, align=left, right] at (   11.0000,  -11.5) {Lower bound\\ by Thm.~\ref{LasPalpha}};

\end{tikzpicture}
\caption{Computational results and lower bounds for $p(n)$.}\label{figPalpha}
\end{center}
\end{figure}

\subsection{The cropped hypercube $Q_{n,\r}$}

Next, we turn our attention to the cropped hypercube
\[
Q_{n,\r} := \set{x \in [0,1]^n : \sum_{i \in S}(1-x_i) + \sum_{i \not\in S}  x_i \geq \r,~\forall S \subseteq [n]}.
\]
Observe that, for every $S \subseteq [n]$, its incidence vector violates the inequality corresponding to $S$ in the description of $Q_{n,\r}$. Thus, we see that $\left( Q_{n,\r} \right)_I = \es$. Independently, Cook and Dash~\cite{CookD01a} and Goemans and the second author~\cite{GoemansT01a} showed that $Q_{n, 1/2}$ has $\LS_+$-rank $n$. Subsequently, the authors showed in~\cite{AuT16a} that the $\SA_+$-rank of $Q_{n,1/2}$ is also $n$. In fact, the results therein readily imply that $\SA_+^k \left( Q_{n, \r} \right)= Q_{n, \r - k/2}$ for all $\r \in (0,1/2]$ and $k \in [n]$. Thus, it follows that $Q_{n,\r}$ has $\SA_+$-rank $n$ for all $\r \in (0, 1/2]$.

\ignore{
\begin{cor}\label{SA+rankneg}
For every $n \geq 2$ and $\r \in \left(0,\frac{1}{2} \right]$, $\SA_{+}^k\left(Q_{n,\r}\right) = Q_{n, (k+\r)/2 }$ for every $k \in [n]$.  In particular, the $\SA_{+}$-rank of $Q_{n,\r}$ is $n$.
\end{cor}

\begin{proof}
Notice that, for all $x \in Q_{n,\r}$ and for all $I,J, I',J' \subseteq [n]$ such that $I \cup J = I' \cup J'$ and $|I| + |J| = k$, $x^I_J \in Q_{n,\r} \iff x^{I'}_{J'} \in Q_{n,\r}$. Thus, it follows from Lemma~\ref{SA_+} that
\[
\SA_+^k(P) \supseteq \bigcap_{I \subseteq [n], |I| = k} \set{x :  x^{I}_{\es} \in P} = Q_{n, (k+\r)/2}.
\]
In fact, the reverse containment holds as well, since it is well known that $\LS_0^k\left(Q_{n,\r}\right) = Q_{n, (k+\r)/2}$, where $\LS_0^k$ is a polyhedral lift-and-project operator defined in~\cite{LovaszS91a} that is dominated by $\SA_+^k$ in general.

Since $\left( Q_{n,\r} \right)_I = \es$, and $\frac{1}{2}\bar{e} \in Q_{n, (k+\r)/2} $ for all $k \leq n-1$, it follows that the $\SA_+$-rank of $Q_{n,\r}$ is $n$.
\end{proof}
}

As for the $\Las$-rank of $Q_{n, 1/2}$, it is shown to be $1$ for $n=2$ in~\cite{Laurent03a}, and $2$ for $n=4$ in~\cite{Cheung07a}. While $\Las$ depends on the algebraic description of the initial relaxation, the following observation significantly simplifies the analysis of the $\Las$-rank of $Q_{n,\r}$.

\begin{prop}\label{LasQalpha}
Suppose $n,k$ are fixed positive integers and $\r \in (0,1)$. Define the vector $w \in \mR^{\A_n^+}$ where
\[
w[S|_1] := (n-|S|-2\r) 2^{-|S|-1},~\forall S \subseteq [n].
\]
Then $\Las^k\left(Q_{n,\r}\right) \neq \es$ if and only if $\M_k(w) \succeq 0$.
\end{prop}

\begin{proof}
Suppose  $\Las^k\left(Q_{n,\r}\right) \neq \es$, and let $Y \in \widehat{\Las}^k\left(Q_{n,\r}\right)$. Notice that every automorphism for the unit hypercube is also an automorphism for $Q_{n,\r}$. If we take these $2^n n!$ automorphisms and apply them onto $Y$ as outlined in the proof of Proposition~\ref{Lassym}, we obtain $2^n n!$ matrices in $\widehat{\Las}^k\left(Q_{n,\r}\right)$. Let $\bar{Y}$ be the average of these matrices. Then by the symmetry of $Q_{n,\r}$, we know that $\bar{Y} = \M_k(y)$, where $y[S|_1] = 2^{-|S|},~\forall S \subseteq [n]$.

By the convexity of $\widehat{\Las}^k\left(Q_{n,\r}\right)$, $\bar{Y} \in \widehat{\Las}^k\left(Q_{n,\r}\right)$, and thus satisfies ($\Las 2$) for all of the $2^n$ equalities defining $Q_{n,\r}$. In fact, due to the entries of $\bar{Y}$, the matrix $\bar{Y}^j$ is the same for all $2^n$ inequalities describing $Q_{n,\r}$. Thus, using the inequality $\sum_{i=1}^n x_i \leq n-\r$ and applying Lemma~
\ref{thetapower} with $\t = \frac{1}{2}$, we obtain that
\[
\bar{Y}^j[S|_1, T|_1] =  (n- |S \cup T| - 2\r) 2^{-|S\cup T|-1} = \M_k(w)[S|_1,T|_1]
\]
for all $S, T \subseteq [n], |S|, |T| \leq k$. Hence, we deduce that $\Las^k\left(Q_{n,\r}\right) \neq \es \Rightarrow \M_k(w) \succeq 0$.

The converse can be proven by tracing the above argument backwards. First, it follows from Lemma~\ref{thetapower} that $\bar{Y} \succeq 0$. Then, again, the matrix  $\bar{Y}^j$ is exactly $\M_k(w)$ for all $2^n$ inequalities describing $Q_{n,\r}$. Since $\M_k(w) \succeq 0$ by assumption, $\bar{Y} \in \widehat{\Las}^k\left(Q_{n,\r}\right)$. Thus, we obtain that $\frac{1}{2} \bar{e} \in \Las^k\left(Q_{n,\r}\right)$, and so $\Las^k\left(Q_{n,\r}\right) \neq \es$.
\end{proof}

Thus, computing the $\Las$-rank of $Q_{n,\r}$ reduces to finding the largest $k$ where the matrix $\M_k(w)$ defined in the statement of Proposition~\ref{LasQalpha} is positive semidefinite (which would then imply that the $\Las$-rank of $Q_{n,\r}$ is $k+1$). Using that, we are able to show the following:

\begin{thm}\label{LasQalpha3}
For every $n \geq 2$, let $q(n)$ be the largest $\r$ where $Q_{n,\r}$ has $\Las$-rank $n$. Then
\[
\frac{n+1}{2^{n+2} - n - 3} \leq q(n) \leq \frac{n}{2^{n+1}-2}.
\]
\end{thm}

\begin{proof}
We first prove the lower bound. If we let $\t = \frac{1}{2}$ in~\eqref{alphabound}, we obtain that $\r \leq \frac{n+1}{2^{n+2} - n - 3}$ implies $\M_{n-1}(w) \succeq 0$ where $w[S|_1] = (n-|S|-2\r) 2^{-|S|-1},~\forall S \subseteq [n]$. Thus, the claim  follows from  Proposition~\ref{LasQalpha}.

As for the upper bound, we show that if $\r >  \frac{n}{2^{n+1}-2}$, then $Q_{n,\r}$ has $\Las$-rank at most  $n-1$. Define $x \in \mR^{\A_n^+}$ where $x[S|_1] := (-2)^{|S|}$ for all $S \subseteq [n]$. Also, let $x'$ denote the vector in $\mR^{\A_{n-1}^+}$ obtained from $x$ by removing the entry corresponding to $[n]|_1$. By Proposition~\ref{LasQalpha}, if we let  $w[S|_1] = (n- |S| - 2\r)2^{-|S|-1}$ for all $S \subseteq [n]$ and show that
 $x'^{\top} \M_{n-1}(w) x' < 0$ whenever $\r > \frac{n}{2^{n+1}-2}$, then $\M_{n-1}(w) \not\succeq 0$, and our claim follows.

Recall that $\M_{n}(w) = Z \Diag(u) Z^{\top}$ where $u[S|_1] = (n-|S|-\r) 2^{-n}$ for all $S \subseteq [n]$. Also, note that
$(Z^{\top}x)[S|_1] = (-1)^{|S|}$  for all $S \subseteq [n]$. Thus,
\begin{eqnarray*}
x^{\top} \M_{n}(w) x  &=& x^{\top} Z \Diag(u) Z^{\top} x \\
&=& \sum_{S \subseteq [n]} \left( (Z^{\top}x)[S|_1] \right)^2 u[S|_1]\\
&=& \sum_{i =0}^{n} \binom{n}{i} \left( (-1)^i \right)^2 (n - i - \r)2^{-n}\\
&=& \left( n2^n - n2^{n-1} - \r 2^n\right) 2^{-n}\\
&=& \frac{n}{2} - \r.
\end{eqnarray*}
On the other hand,
\begin{eqnarray*}
x^{\top} \M_{n}(w) x  &=& \begin{pmatrix} x'^{\top} & (-2)^n \end{pmatrix}
\begin{pmatrix} \M_{n-1}(w) & -2^{-n} \r \bar{e} \\ -2^{-n} \r \bar{e}^{\top} & -2^{-n} \r \end{pmatrix}
\begin{pmatrix} x' \\ (-2)^n \end{pmatrix}\\
&=& \begin{pmatrix} x'^{\top} \M_{n-1}(w) + (-1)^{n+1}\r \bar{e}^{\top} &  -2^{-n} \r x'^{\top}\bar{e} + (-1)^{n+1} \r \end{pmatrix}
\begin{pmatrix} x' \\ (-2)^n \end{pmatrix}\\
&=& x'^{\top} \M_{n-1}(w) x' + (-1)^{n+1}\r \bar{e}^{\top}x' + (-1)^{n+1} \r x'^{\top}\bar{e} + (-1)^{n+1}(-2)^n \r \\
&=& x'^{\top}\M_{n-1}(w) x' +  \left( 2^n-2 \right) \r.
\end{eqnarray*}
(It is helpful to observe that $\bar{e}^{\top} x' = (-1)^n - (-2)^n$.) Hence, we combine the above and obtain that
\[
x'^{\top} \M_{n-1}(w) x' = \frac{n}{2} - (2^n-1) \r,
\]
which is negative whenever $\r > \frac{n}{2^{n+1}-2}$. This finishes the proof.
\end{proof}

Therefore, akin to what Cheung showed for $P_{n,\r}$, there does not exist a fixed $\r$ where $Q_{n,\r}$ has $\Las$-rank $n$ for all $n$. Also, as with $P_{n,\r}$, the $\Las$-rank of $Q_{n,\r}$ varies under the choice of $\r$. For instance, Figure~\ref{fig74} illustrates the $\Las$-rank for $Q_{n, \l/1000}$ for $\l \in [500]$ and several values of $n$. The pattern is similar for all other values of $n$ we were able to test --- the $\Las$-rank is around $\frac{n}{2}$ when $\r = \frac{1}{2}$, and slowly rises to $n$ as $\r$ approaches $0$.  Recently, related to the Figure~\ref{fig74}, Kurpisz, Lepp\"{a}nen and Mastrolilli \cite{KurpiszLM2016} proved that the Lasserre rank of $Q_{n, 1/2}$ is between $\Omega(\sqrt{n})$ and $n - \Omega( n^{1/3} )$.

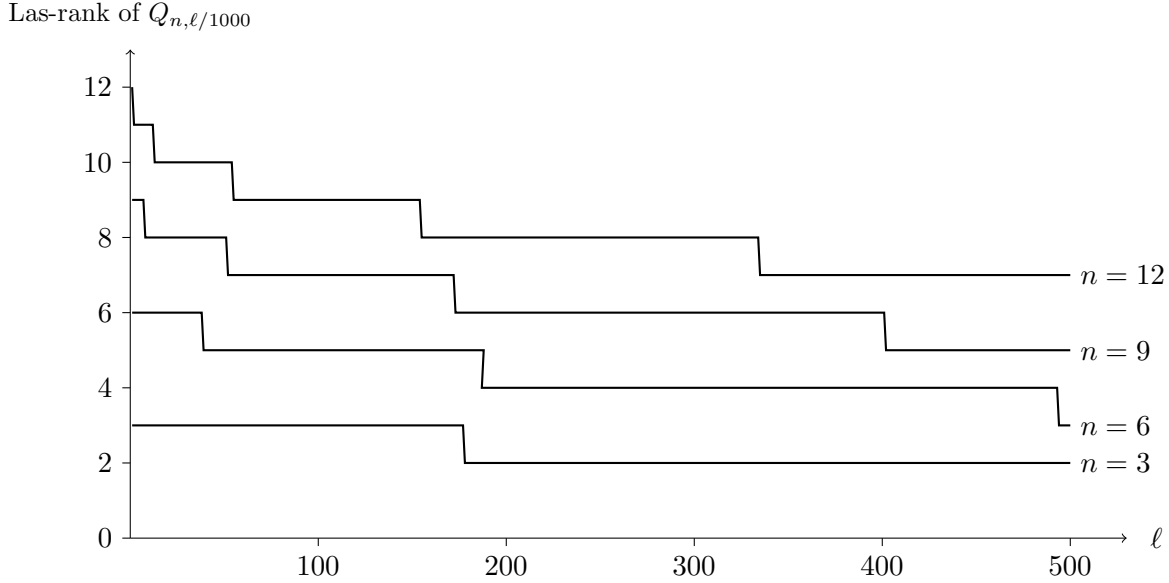
\begin{figure}[htb]
\begin{center}

\begin{tikzpicture}[y=0.5cm, x=.025cm]
	\draw[->] (0,0) -- coordinate (x axis mid) (530,0);
    	\draw[->] (0,0) -- coordinate (y axis mid) (0,13);
    	\foreach \x in {100,200,...,500}
     		\draw (\x,0) -- (\x,-3pt)
			node[anchor=north] {\x};
    	\foreach \y in {0,2,...,12}
     		\draw (0,\y) -- (-3pt,\y)
     			node[anchor=east] {\y};
	\node[right = 0.18cm] at (530,0) {$\l$};
	\node[above=0.12cm] at (0,13) {{\small $\Las$-rank of $Q_{n, \l/1000}$}};

\ignore{
\node[draw, circle, inner sep=0pt, minimum size = 0.1cm, fill] at (282,9) {};
\node[below] at (282,9) {$(282,9)$};
\node[draw, circle, inner sep=0pt, minimum size = 0.1cm, fill] at (137,10) {};
\node[below] at (137,10) {$(137,10)$};
\node[draw, circle, inner sep=0pt, minimum size = 0.1cm, fill] at (54,11) {};
\node[below] at (54,11) {$(54,11)$};
\node[draw, circle, inner sep=0pt, minimum size = 0.1cm, fill] at (16,12) {};
\node[above right] at (16,12) {$(16,12)$};
\node[draw, circle, inner sep=0pt, minimum size = 0.1cm, fill] at (3,13) {};
\node[below left] at (3,13) {$(3,13)$};
\node[draw, circle, inner sep=0pt, minimum size = 0.1cm, fill] at (1,14) {};
\node[above right] at (1,14) {$(1,14)$};
}


\node[right] at (500,7) {$n=12$};
\node[right] at (500,5) {$n=9$};
\node[right] at (500,3) {$n=6$};
\node[right] at (500,2) {$n=3$};

	\draw[thick] (1,12) -- (2,11) -- (12,11) -- (13,10) -- (54,10) -- (55,9) --  (154,9) -- (155,8) -- (334,8)-- (335,7) -- (500,7);
	\draw[thick] (1,9) -- (7,9) -- (8,8) -- (51,8) -- (52,7) -- (172,7) -- (173,6) -- (401,6)-- (402,5)-- (500,5);
	\draw[thick] (1,6) -- (38,6) -- (39,5) -- (188,5) -- (187,4) -- (493,4)-- (494,3)-- (500,3);
	\draw[thick] (1,3) -- (177,3) -- (178,2) -- (500,2);


\end{tikzpicture}
\caption{The $\Las$-rank of $Q_{n,\r}$ for varying values of $\r:=\l/1000$, for $n \in \set{3,6,9,12}$.}\label{fig74}
\end{center}
\end{figure}

Also, recall that we let $q(n)$ be the largest $\r$ where $Q_{n,\r}$ has $\Las$-rank $n$. It follows from Theorem~\ref{LasQalpha3} that $\frac{2^{n+1}}{n q(n)}$ is roughly bounded between $\frac{1}{2}$ and $1$. Note that since $\Las$ imposes as many positive semidefiniteness constraints as there are defining inequalities for the given relaxation (which there are exponentially many for $Q_{n,\r}$), the relaxation $\Las^k\left(Q_{n,\r}\right)$ is not obviously tractable, even when $k$ is a constant. Now, computing $q(n)$ requires verifying whether $\Las^{n-1}\left(Q_{n,\r}\right)$ is empty, which by definition of $\Las$ is the projection of $\widehat{\Las}^{n-1}\left(Q_{n,\r}\right)$, a set of matrices of order $\Omega(2^n) \times \Omega(2^n)$ with $\Omega(2^n)$ positive semidefiniteness constraints. Instead of solving the feasibility problem of such a large number of variables and constraints, Proposition~\ref{LasQalpha} uses the symmetries of $Q_{n,\r}$ (as well as the fact that $\Las$ preserves symmetries and commutes with all automorphisms of the unit hypercube, as shown in Proposition~\ref{Lassym}) to reduce this task to  checking the positive semidefiniteness of $\M_{n-1}(w)$,  a $(2^n-1) \times (2^n-1)$ matrix with known entries. Furthermore, notice that if $\M_{n-1}(w)$ had an eigenvector $x$ with negative eigenvalue, we could assume that $x[S|_1] = x[T|_1]$ whenever $|S| = |T|$, due to the symmetries of the entries in $\M_{n-1}(w)$. Hence, if we define the $n$-by-$n$ matrix $W$ whose rows and columns are indexed by $\set{0,1,\ldots, n-1}$ such that
\begin{eqnarray*}
W[i,j] &:=& \sum_{S,T \subseteq [n], |S|=i, |T|=j} \M_{n-1}(w)[S|_1, T|_1]\\
&=& 2^{-i-j-1} n\binom{n-1}{i}\binom{n-1}{j} \left( \sum_{k=0}^{n-1} \frac{\binom{i}{k} \binom{j}{k}}{\binom{n-1}{k}} \right)  - \r 2^{-i-j} \binom{n}{i}\binom{n}{j} \left( \sum_{k=0}^{n} \frac{\binom{i}{k} \binom{j}{k}}{\binom{n}{k}} \right),
\end{eqnarray*}
then it follows that $\M_{n-1}(w) \succeq 0$ if and only if $W \succeq 0$. This reduction allows us to verify if $Q_{n,\r}$ has $\Las$-rank $n$ by simply checking if the $n$-by-$n$ matrix $W$ is positive semidefinite. Using the reduction above, we computed $\frac{2^{n+1}}{n q(n)}$ to within two decimal places for $n \in \set{2,3,\ldots, 16}$, as illustrated in Figure~\ref{fig73}. 

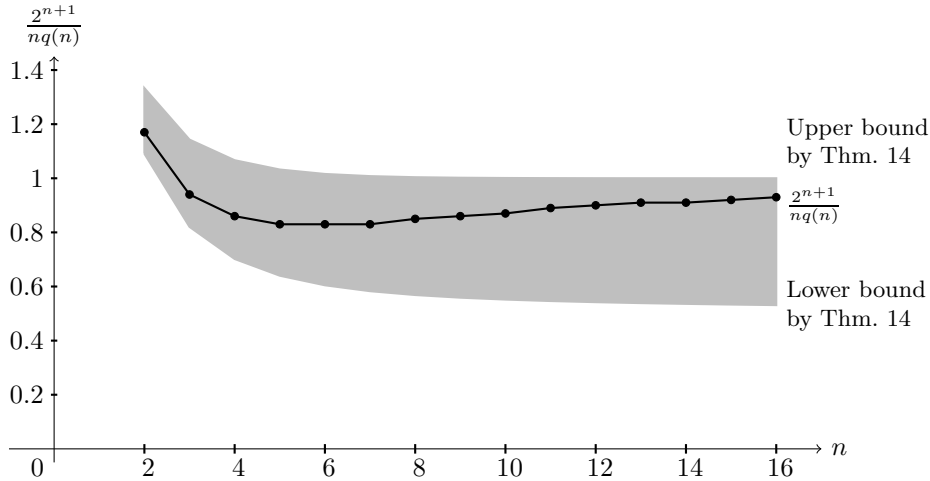
\begin{figure}[htb]
\begin{center}

\begin{tikzpicture}[scale = 0.6, xscale=1,yscale=6, font=\small,  word node/.style={font=\footnotesize}]]

\def\xlb{0};
\def\xub{16};
\def\ylb{0};
\def\yub{1.4};
\def\xbuf{1};
\def\ybuf{0.05};

\filldraw[thick,lightgray]

(   2.0000  ,  1.0909)--(
    3.0000  ,  0.8205)--(
    4.0000  ,  0.7018)--(
    5.0000  ,  0.6400)--(
    6.0000  ,  0.6046)--(
    7.0000  ,  0.5828)--(
    8.0000  ,  0.5686)--(
    9.0000  ,  0.5588)--(
   10.0000  ,  0.5518)--(
   11.0000  ,  0.5464)--(
   12.0000  ,  0.5422)--(
   13.0000  ,  0.5387)--(
   14.0000  ,  0.5359)--(
   15.0000  ,  0.5334)--(
   16.0000  ,  0.5313)--(
   16.0000 ,   1.0000)--(
   16.0000 ,   1.0000)--(
   15.0000  ,  1.0000)--(
   14.0000   , 1.0001)--(
   13.0000   , 1.0001)--(
   12.0000   , 1.0002)--(
   11.0000   , 1.0005)--(
   10.0000   , 1.0010)--(
    9.0000   , 1.0020)--(
    8.0000   , 1.0039)--(
    7.0000   , 1.0079)--(
    6.0000   , 1.0159)--(
    5.0000   , 1.0323)--(
    4.0000   , 1.0667)--(
    3.0000   , 1.1429)--(
  2.0000   , 1.3333)--(
(   2.0000  ,  1.0909);

\draw [->] (\xlb - \xbuf,0) -- (\xub + \xbuf,0);
\draw [->] (0, \ylb - \ybuf) -- (0, \yub + \ybuf);

\foreach \x in {\xlb ,2, ...,\xub}
{
    \ifthenelse{\NOT 0 = \x}{\draw[thick](\x ,-0.5pt) -- (\x ,0.5pt);}{}
    \ifthenelse{\NOT 0 = \x}{\node[anchor=north] at (\x,0) (label) {{ $\x$}};}{}
}

\foreach \y in { 0.2 ,0.4,0.6,0.8, 1,1.2,1.4}
{
\draw[thick](-2pt, \y) -- (2pt, \y);
\node[anchor=east] at (0,\y) (label) {{ $\y$}};

}

\node[anchor=north east] at (0,0) (label3) { $0$};

\node[anchor=west] at (\xub + \xbuf,0) (label3) {$n$};
\node[anchor=south] at (0, \yub+\ybuf) (label3) {$\frac{2^{n+1}}{n q(n)} $};

%
\draw[thick](
   2.0000 ,   1.1700)--(
    3.0000 ,   0.9400)--(
    4.0000 ,   0.8600)--(
    5.0000 ,   0.8300)--(
    6.0000 ,   0.8300)--(
    7.0000 ,   0.8300)--(
    8.0000 ,   0.8500)--(
    9.0000 ,   0.8600)--(
   10.0000,    0.8700)--(
   11.0000,    0.8900)--(
   12.0000 ,   0.9000)--(
   13.0000  ,  0.9100)--(
   14.0000  ,  0.9100)--(
   15.0000  ,  0.9200)--(
   16.0000  ,  0.9300
);

\def\y{0.7};

\foreach \position in {(
   2.0000 ,   1.1700),(
    3.0000 ,   0.9400),(
    4.0000 ,   0.8600),(
    5.0000 ,   0.8300),(
    6.0000 ,   0.8300),(
    7.0000 ,   0.8300),(
    8.0000 ,   0.8500),(
    9.0000 ,   0.8600),(
   10.0000,    0.8700),(
   11.0000,    0.8900),(
   12.0000 ,   0.9000),(
   13.0000  ,  0.9100),(
   14.0000  ,  0.9100),(
   15.0000  ,  0.9200),(
   16.0000  ,  0.9300)
}
{
\node[draw, circle, inner sep=0pt, minimum size = 0.1cm, fill] at \position {};
}

\node[word node, align=left, right] at (16,0.9) {$\frac{2^{n+1}}{n q(n)}$};
\node[word node, align=left, above right] at (16,1) {Upper bound\\ by Thm.~\ref{LasQalpha3}};
\node[word node, align=left, right] at (16,0.53) {Lower bound\\ by Thm.~\ref{LasQalpha3}};

\end{tikzpicture}
\caption{Computational results and possible ranges for $q(n) := \min\set{\r : \Las^{n-1}\left(Q_{n,\r}\right) \neq \es}$.}\label{fig73}
\end{center}
\end{figure}

As for the $\BZ_+'$-rank of $Q_{n,\r}$, it was shown in~\cite{BienstockZ04a} that $Q_{n,1/2}$ has $\BZ$-rank $2$, where $\BZ$ is a polyhedral operator dominated by $\BZ_+$ and $\BZ_+'$. Thus, it follows that the $\BZ_+'$-rank of $Q_{n,1/2}$ is at most $2$. However, we remark that, as with the Lasserre operator, the Bienstock--Zuckerberg operators also require an explicitly given system of inequalities for the input set. In particular, the run-time of these operators depends on the size of the system (which, again, is exponential in $n$ in the case of $Q_{n,\r}$). Thus, $\BZ^k \left( Q_{n,\r} \right)$ is not obviously tractable, even for $k = O(1)$. On the other hand, operators such as $\SA_+, \SA$ and $\BCC$ are able to produce tightened relaxations that are tractable as long as we have an efficient separation oracle of the input set (which does exist for the cropped hypercube ---  note that $x \in Q_{n,\r}$ if and only if $x \in [0,1]^n$ and  $\sum_{i=1}^n | x_i - \frac{1}{2} | \leq \frac{n}{2} - \r$).

\section{Integrality gaps of lift-and-project relaxations}

We conclude this paper by noting some interesting tendencies of the integrality gaps of some lift-and-project relaxations. First, given a compact, convex set $P \subseteq [0,1]^n$ where $P_I \neq \es$ and vector $c \in \mathbb{R}^{n}$, the \emph{integrality gap} of $P$ with respect to $c$ is defined to be
\[
\gamma_c(P) := \frac{ \max\set{ c^{\top}x : x \in P}}{ \max\set{ c^{\top}x : x \in P_I}}.
\]
The integrality gap gives a measure of how ``tight'' the relaxation $P$ is in the objective function direction of $c$. Here, we show that the integrality gap of $P_{n,\r}$ with respect to the all-ones is invariant under $k$ iterations of several different operators.

.

\begin{thm}\label{SA+gap}
For every integer $n \geq 2$, for every $\r \in (0,1)$ and for every operator $\Gamma \in \set{\tilde{\LS}, \LS_+, \SA, \SA_+}$, we have
\[
\gamma_{\bar{e}}\left(\Gamma^{k} \left( P_{n,\r} \right) \right) = 1 + \frac{(n-k)(1-\r)}{(n-1)(n-k+k\r)},
\]
for every $k \in \set{0,1,2,\ldots, n}$.
\end{thm}

\begin{proof}
We prove our claim by showing that
\begin{equation}\label{SA+gap1}
\max \set{ \t : \t\bar{e} \in \tilde{\LS}^k\left( P_{n,\r} \right) } \geq  \frac{n-k + (k-1)\r}{n-k+k\r} \geq \max \set{ \t : \t\bar{e} \in \SA_+^k\left( P_{n,\r} \right) }.
\end{equation}
Then the result follows from the dominance relationships between the operators. First, the claim is obvious when $k=0$ or when $k=n$, and thus from here on we assume that $k \in [n-1]$. Let $P := P_{n,\r}$.  We first prove the first inequality in~\eqref{SA+gap1}.

Given $\t \bar{e} \in \tilde{\LS}^k(P)$, we know that there exist coefficients $a_T$ and vectors $v_T \in [0,1]^{n-k}$ for each $T \subseteq [k]$ where
\[
\t \bar{e} = \sum_{T \subseteq [k]} a_T \begin{pmatrix} \chi_T \\ v_T \end{pmatrix}.
\]
(Here, $\chi_T$ is the incidence vector of $T$ in $\set{0,1}^k$.) Note that the operator $\tilde{\LS}$ requires that the $a_T$'s be nonnegative and sum up to $1$. Also, note that $\begin{pmatrix} \chi_T \\ v_T \end{pmatrix}$ is in $P$ for all $v_T \in [0,1]^{n-k}$ whenever $|T| < k$. For $T = [k]$, the constraint $\bar{e}^{\top}x \leq n- \r$ implies that $\bar{e}^{\top} v_T \leq n-k - \r$.

Due to the symmetry of $P$, given one convex combination of $\t \bar{e}$, we could obtain many other convex combinations by applying any permutation on $[n]$ that fixes $[k]$. If we take the average of all these combinations, we would obtain a ``symmetric'' one where $a_T = a_T'$ and $v_T = v_T'$ whenever $|T| = |T'|$, and that $v_T$'s are all multiples of the all-ones vector. Thus, we may further assume that there are nonnegative real numbers $a_i, v_i, i \in \set{0,1, \ldots, k}$ where
\begin{equation}\label{SA+gap8}
\t \bar{e} = \sum_{i=0}^{k}  a_i \begin{pmatrix} \frac{i}{k} \bar{e} \\ v_i \bar{e} \end{pmatrix},
\end{equation}
such that the $a_i$'s sum to $1$, $ 0 \leq v_{i} \leq 1$ for all $i < k$, and $0 \leq v_{k} \leq \frac{n- k-\r}{n-k}$. Thus,~\eqref{SA+gap8} is equivalent to saying that the point $(\t, \t) \in \mR^2$ is a convex combination of the points in the sets $\set{ (\frac{i}{k}, v_i) : i \in \set{0, 1, \ldots, k-1}, 0 \leq v_i \leq 1}$ and $\set{(1, v_i) : 0 \leq v_i \leq \frac{n-k-\r}{n-k}}$. It is easy to see that the convex hull of these points in $\mR^2$ form the polytope illustrated in Figure~\ref{SA+gapfig}.

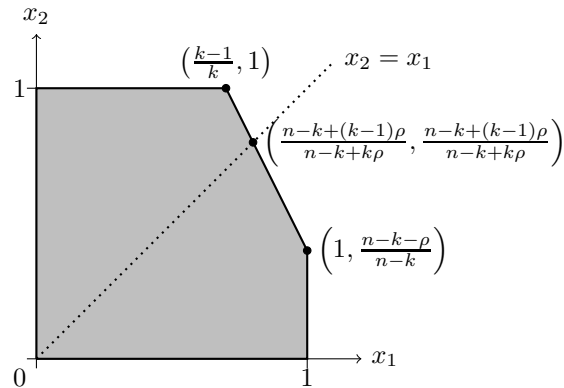
\begin{figure}[htb]
\begin{center}
\begin{tikzpicture}[scale = 1.2, y=3cm, x=3cm, font=\small]

\def\xend{0.7}
\def\yend{0.4}

	\draw[->] (0,0) -- coordinate (x axis mid) (1.2,0);
    	\draw[->] (0,0) -- coordinate (y axis mid) (0,1.2);

\foreach \x in {0,...,1}
     		\draw (\x,1pt) -- (\x,-3pt);
    	\foreach \y in {0,...,1}
     		\draw (1pt,\y) -- (-3pt,\y) ;

\node[anchor=north] at (1,0) (label2) {1};
\node[anchor=east] at (0,1) (label3) {1};
\node[anchor=west] at (1.2,0) (label2) {$x_1$};
\node[anchor=south] at (0,1.2) (label3) {$x_2$};

\node[anchor=north east] (0,0) {0};

\draw[fill = lightgray, thick] (0,0) -- (1,0) -- (1,\yend) -- (\xend,1) -- (0,1)-- (0,0);

\draw[dotted, thick] (0,0) -- (1.1,1.1);
\node[anchor=west] at (1.1,1.1) (label) {$x_2=x_1$};


\node[draw, circle, inner sep=0pt, minimum size = 0.1cm, fill] at (1,\yend) {};
\node[right] at (1, \yend) {$\left(1,\frac{n-k-\r}{n-k}\right)$};
\node[draw, circle, inner sep=0pt, minimum size = 0.1cm, fill] at (\xend,1) {};
\node[above] at (\xend,1) {$\left(\frac{k-1}{k},1\right)$};

\node[draw, circle, inner sep=0pt, minimum size = 0.1cm, fill] at ( {  (1 - \xend * \yend)/(2 - \xend - \yend) }, {  (1 - \xend * \yend)/(2 - \xend - \yend) }) {};
\node[right] at ({  (1 - \xend * \yend)/(2 - \xend - \yend) }, {  (1 - \xend * \yend)/(2 - \xend - \yend) }) {$\left( \frac{n-k+(k-1)\r}{n-k+k\r},\frac{n-k+(k-1)\r}{n-k+k\r}\right)$};

\end{tikzpicture}
\end{center}
\caption{Reduction of finding $\max\set{\t : \t\bar{e} \in \tilde{\LS}^k\left( P_{n,\r} \right)}$ to two dimensions.}\label{SA+gapfig}
\end{figure}

Then it is easy to see that the largest $\t$ where $(\t,\t)$ is contained in the convex hull is obtained by the convex combination
\[
\frac{n-k}{n-k+k\r} \begin{pmatrix} 1 \\ \frac{n-k-\r}{n-k} \end{pmatrix} + \frac{k\r}{n-k+k\r} \begin{pmatrix} \frac{k-1}{k} \\ 1 \end{pmatrix}  = \begin{pmatrix} \frac{n-k+(k-1)\r}{n-k+k\r} \\ \frac{n-k+(k-1)\r}{n-k+k\r} \end{pmatrix}.
\]
This establishes the upper bound on $\t$.

Next, we turn to show the second inequality in~\eqref{SA+gap1} by proving that $\frac{n-k+(k-1)\r}{n-k+k\r} \bar{e} \in \SA_+^k(P)$ for every $k$. First, define $y \in \A_n^+$ where $y[S|_1] := 1 - \frac{|S|\r}{n-k+k\r}$ for every $S \subseteq [n]$. We first show that $\M_n(y) \succeq 0$.
By Lemma~\ref{zetau}, we know that $\M_n(y) = Z \Diag(u) Z^{\top}$ where $u$ is the vector with entries
\begin{eqnarray*}
u[S|_1] &=& \sum_{T \supseteq S} (-1)^{|T\sm S|} y[S|_1] = \sum_{j=0}^{n-|S|} \binom{n-|S|}{j} (-1)^{j} \left(1 - \frac{ (|S|+j)\r}{n-k+k\r} \right)\\
&=&  \left\{
\begin{array}{ll}
0 & \tn{if $|S| \leq n-2$;}\\
\frac{\r}{n-k+k\r} & \tn{if $|S| = n-1$;}\\
1 - \frac{n\r}{n-k+k\r} & \tn{if $S = [n]$.}
\end{array}
\right.
\end{eqnarray*}
Note that $1 - \frac{n\r}{n-k+k\r} \geq 0 \iff (n-k)(\r-1) \leq 0$, which does hold as $n \geq k$ and $\r < 1$. Hence, since $u \geq 0$, we deduce that $\M_n(y) \succeq 0$, and in particular $\M_k(y) \succeq 0$.

Next, define $L \in \mR^{\A_{k} \times \A_{k}^+}$ where
\[
L[S|_1 \cap T|_0, U|_1] := \left\{
\begin{array}{ll}
(-1)^{|S|}& \tn{if $S \cup T = U$;}\\
0 & \tn{otherwise,}
\end{array}
\right.
\]
and let $Y := L \M_k(y) L^{\top}$. We claim that $Y \in \widehat{\SA}_+^{k}(P)$. First, $\M_k(y) \succeq 0$ implies that $Y \succeq 0$. Also, ($\SA_+ 1$) holds as $Y[\F,\F] = 1$, and it is not hard to see that $Y \geq 0$, as every entry in $Y$ is either $0, \frac{\r}{ n - k +k\r}$ or $1 -  \frac{i\r}{ n -k+k\r}$ for some integer $i \in \set{0,\ldots, n}$. Next, we check that $\het{x}(Y e_{\b}) \in K(P)$ for all $\b \in \A_{k}$. Given $\b = S|_1 \cap T|_0, \het{x}(Y e_{\b})$ is the zero vector whenever $|T| \geq 2$, and is the vector $\frac{\r}{n\r + 1 - \r} (\bar{e} - e_i)$ whenever $T = \set{i}$ for some $i \in [n]$. In both cases, $\sum_{i=1}^n Y[i|_1, \b] \leq (n-\r) Y[\F,\b]$ easily holds.

Finally, suppose $\b = S|_1$ for some $S \subseteq [n]$ where $|S| \leq k$. Observe that
\[
Y[\a, S|_1] = \left\{
\begin{array}{ll}
\frac{n-k+(k-|S|)\r}{n-k+k\r} & \tn{if $\a=\F$ or $\a =i|_1$ where $i \in S$;}\\
 \frac{n-k-(k-|S|-1)\r}{n-k+k\r} & \tn{otherwise.}
\end{array}
\right.
\]
Now
\begin{eqnarray*}
\sum_{i=1}^n Y[i|_1, S|_1]  &=&  |S| \left( \frac{n-k+(k-|S|)\r}{n-k+k\r}\right) +  (n-|S|) \left( \frac{n-k+(k-|S|-1)\r}{n-k+k\r}\right)\\
&\leq&  (n-\r)\left(  \frac{n-k+(k-|S|)\r}{n-k+k\r}\right)\\
&=& (n-\r) Y[\F,S|_1].
\end{eqnarray*}
Thus, $\het{x}(Ye_{\b}) \in K(P)$ in this case as well. Finally, it is not hard to see that the entries of $Y$ satisfy ($\SA_+ 3$), ($\SA_+4$) and ($\SA_+ 5$). This completes our proof.
\end{proof}

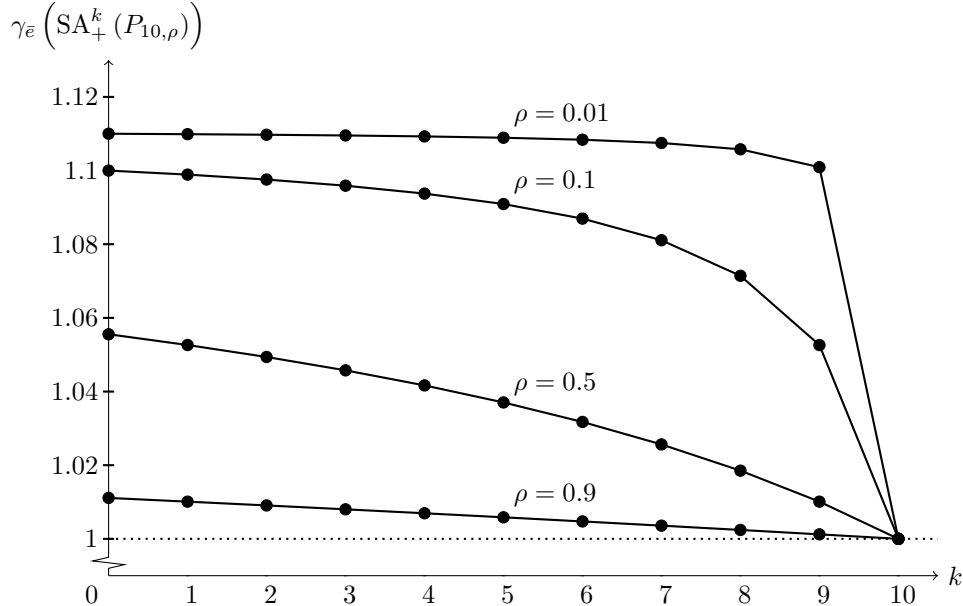
\begin{figure}[htb]
\begin{center}

\begin{tikzpicture}[scale =0.7, xscale=1.5,yscale=0.7,  font=\small, word node/.style={font=\small}]]

\def\n{10};

\def\xlb{0};
\def\xub{\n};
\def\ylb{0};
\def\yub{13};
\def\xbuf{0.5};
\def\ybuf{1};

\draw [->] (0,0) -- (\xub + \xbuf,0);
\draw [->] (0,0) -- (0,0.2) -- (0.2,0.3) -- (-0.2, 0.4) -- (0,0.5) -- (0,\yub + \ybuf);

\foreach \x in {\xlb ,1, ...,\xub}
{
    \ifthenelse{\NOT 0 = \x}{\draw[thick](\x ,-0.5pt) -- (\x ,0.5pt);}{}
    \ifthenelse{\NOT 0 = \x}{\node[anchor=north] at (\x,0) (label) {{ $\x$}};}{}
}

\foreach \y in {1,1.02, 1.04, 1.06, 1.08,1.1,1.12}
{
\draw[thick](-2pt, { 1+ 100*(\y-1)} ) -- (2pt,{ 1+ 100*(\y-1)});
\node[anchor=east] at (0,{ 1+ 100*(\y-1)}) (label) {{ $\y$}};

}

\node[anchor=north east] at (0,0) (label3) { $0$};

\node[anchor=west] at (\xub + \xbuf,0) (label3) {$k$};
\node[anchor=south] at (0, \yub+\ybuf) (label3) {$\gamma_{\bar{e}}\left( \SA_+^k \left(P_{10,\r} \right) \right)$};

%

\node[anchor = west] at (5,12.5) {$\r = 0.01$};
\node[anchor = west] at (5,10.7) {$\r = 0.1$};
\node[anchor = west] at (5,5.2) {$\r = 0.5$};
\node[anchor = west] at (5,2.2) {$\r = 0.9$};

\def\alpha{0.01}
\draw[thick, domain= 0 : {\n}, samples = 11] plot (\x, {1 + 100*(\n - \x) * ( 1-\alpha) / (\n - 1) / (\n - \x + \x * \alpha)});

\foreach \x in {\xlb ,1, ...,\xub}
{
\node[draw, circle, inner sep=0pt, minimum size = 0.15cm, fill] at ({\x}, {1 + 100*(\n - \x) * ( 1-\alpha) / (\n - 1) / (\n - \x + \x * \alpha)})  {};
}

\def\alpha{0.1}
\draw[thick, domain= 0 : {\n}, samples = 11] plot (\x, {1 + 100*(\n - \x) * ( 1-\alpha) / (\n - 1) / (\n - \x + \x * \alpha)});

\foreach \x in {\xlb ,1, ...,\xub}
{
\node[draw, circle, inner sep=0pt, minimum size = 0.15cm, fill] at ({\x}, {1 + 100*(\n - \x) * ( 1-\alpha) / (\n - 1) / (\n - \x + \x * \alpha)})  {};
}

\def\alpha{0.5}
\draw[thick, domain= 0 : {\n}, samples = 11] plot (\x, {1 + 100*(\n - \x) * ( 1-\alpha) / (\n - 1) / (\n - \x + \x * \alpha)});

\foreach \x in {\xlb ,1, ...,\xub}
{
\node[draw, circle, inner sep=0pt, minimum size = 0.15cm, fill] at ({\x}, {1 + 100*(\n - \x) * ( 1-\alpha) / (\n - 1) / (\n - \x + \x * \alpha)})  {};
}

\def\alpha{0.9}
\draw[thick, domain= 0 : {\n}, samples = 11] plot (\x, {1 + 100*(\n - \x) * ( 1-\alpha) / (\n - 1) / (\n - \x + \x * \alpha)});

\foreach \x in {\xlb ,1, ...,\xub}
{
\node[draw, circle, inner sep=0pt, minimum size = 0.15cm, fill] at ({\x}, {1 + 100*(\n - \x) * ( 1-\alpha) / (\n - 1) / (\n - \x + \x * \alpha)})  {};
}

\draw[thick, dotted] (0,1) -- ( {\xub+\xbuf}, 1);
\end{tikzpicture}
\caption{Integrality gaps of $\SA_+^k \left(P_{10,\r} \right)$ for various $k$ and $\r$.}\label{fig6}
\end{center}
\end{figure}

Figure~\ref{fig6} illustrates the integrality gaps of $\SA_+^k \left( P_{n,\r} \right)$ for various values of $k$ and $\r$ in the case $n=10$ (the behaviour is similar for other values of $n$). In general, when $\r$ is close to $1$, the gap decreases at an almost-linear rate towards $1$. On the other hand, when $\r$ is small, the integrality gap of $\SA_+^k\left(P_{n,\r}\right)$ stays relatively close to $1 + \frac{1}{n-1}$ as $k$ increases to $n-1$, and then abruptly drops to $1$ at the $n^{\tn{th}}$ iteration, where we obtain the integer hull. Again, it follows from Theorem~\ref{SA+gap} that these gaps would be identical if we replaced $\SA_+$ by any operator $\Gamma$ where $\SA_+$ dominates $\Gamma$ and $\Gamma$ dominates $\tilde{\LS}$.


We also note that Theorem~\ref{SA+gap} implies Proposition~\ref{Kndstrong}. Moreover, the techniques used for proving the first inequality in~\eqref{SA+gap1} can be extended to compute $\max\set{\t : \t \bar{e} \in \tilde{\LS}^k \left( P_{n,\r} \right)}$ for any non-integer $\r \in (0,n)$, which would imply Proposition~\ref{GammaLB}.

While the integrality gap for $Q_{n,\r}$ is undefined (as its integer hull is empty for all $\r >0$), we see a similar distinction between its $\SA_+$ and $\Las$ relaxations. Note that since all lift-and-project operators we have studied preserve containment, starting with a tighter initial relaxation might offer a lift-and-project operator a head start and yield stronger relaxations in fewer iterations. However, in the case of $Q_{n,\r}$, different lift-and-project operators utilize this head start in different ways. As mentioned earlier, we know that $\SA_+^k\left(Q_{n,\r}\right) = Q_{n, \r-k/2}$ for all $\r \in (0,1/2]$ and for all $k = [n]$. Thus, given $\r, \r'$ where $0 < \r < \r' \leq \frac{1}{2}$,
\[
\SA_+^k\left(Q_{n,\r}\right) =  Q_{n, \left(\r + k/2 \right)} \supset  Q_{n, \left(\r' + k/2 \right)} = \SA_+^k\left(Q_{n,\r'}\right),
\]
for all $k \in [n-1]$. However, they still converge to the integer hull in the same number of steps. On the other hand, as shown in Theorem~\ref{LasQalpha} and Figure~\ref{fig74}, starting with a larger $\r$ can help $\Las$ arrive at the integer hull in fewer iterations, similar to what we saw with $P_{n,\r}$.

Thus, at least in the case of $P_{n,\r}$ and $Q_{n,\r}$ where $\r \in (0,1)$, all aforementioned operators that are no stronger than $\SA_+$ perform pretty much equally poorly, while deploying $\Las$ does achieve some tangible improvements in rank (at least when $\r$ is not extremely small). Granted, since the number of inequalities imposed by most lift-and-project methods are superpolynomial in $n$ after $\Omega(\log(n))$ rounds, an operator managing to return the integer hull in, say, $\Omega(\sqrt{n})$ iterations is already exerting exponential effort. In that case, claiming that this  operator performs better than another that requires (say) $\Omega(n)$ rounds is somewhat a moot point in practice, at the time of this writing.

Of course, there do exist examples where a stronger lift-and-project operator manages to return a tractable relaxation and outperforms exponential effort by a weaker operator: We showed in Propositions~\ref{SA+PalphaUB} and~\ref{GammaLB} that when $\r = n-O(1)$, $\SA_+$ would return the integer hull in $O(1)$ iterations, while $\tilde{\LS}$ requires  $\Omega(n)$ rounds. Another such instance is the following: Given a graph $G = (V,E)$, consider its fractional stable set polytope, which is defined as
\[
\FRAC(G) = \set{ x \in [0,1]^V : x_i + x_j \leq 1,~\forall \set{i,j} \in E}.
\]
When $G$ is the complete graph on $n$ vertices, it is well known that for hierarchies of polyhedral lift-and-project relaxations (including $\SA$), the integrality gap (with respect to $\bar{e}$) starts at $\frac{n}{2}$, then gradually decreases, and reaches $1$ after $\Omega(n)$ iterations. On the other hand, it takes semidefinite operators such as $\LS_+, \SA_+$ and $\Las$ exactly one iteration to reach the stable set polytope of $K_n$, and thus the corresponding integrality gaps for these operators would dive from $\frac{n}{2}$ to $1$ in just one iteration.

This raises the natural question of whether, in general, there is some efficient way where we could diagnose a given problem and determine the ``best'' lift-and-project method for the job. One step in that direction is through studying how various methods perform on different problem classes.  Such studies would hopefully provide us better guidance on when it is worthwhile to apply an operator that is more powerful but has a higher per-iteration computational cost.

To take this point further, perhaps one could build a shape-shifting operator that adapts to the given problem in some way. Bienstock and Zuckerberg~\cite{BienstockZ04a} devised the first operators that generate different variables for different relaxations (or even different algebraic descriptions of the same relaxation). They showed that this flexibility can be very useful in attacking relaxations of some set covering problems. Thus, perhaps tight relaxations for other hard problems can be found similarly by building a lift-and-project operator with suitable adaptations.

\bibliographystyle{alpha}
\bibliography{ref}

\begin{thebibliography}{KLM16}

\bibitem[AT16]{AuT16a}
Yu~Hin Au and Levent Tun{\c{c}}el.
\newblock A comprehensive analysis of polyhedral lift-and-project methods.
\newblock {\em SIAM Journal on Discrete Mathematics}, 30(1):411--451, 2016.

\bibitem[Au14]{Au14a}
Yu~Hin Au.
\newblock {\em A Comprehensive Analysis of Lift-and-Project Methods for
  Combinatorial Optimization}.
\newblock PhD thesis, University of Waterloo, 2014.

\bibitem[Bal74]{Balas1974}
Egon Balas.
\newblock Disjunctive programming: properties of the convex hull of feasible
  points.
\newblock {\em Management Sciences Research Report}, (348), 1974.

\bibitem[Bal98]{Balas98a}
Egon Balas.
\newblock Disjunctive programming: properties of the convex hull of feasible
  points.
\newblock {\em Discrete Appl. Math.}, 89(1-3):3--44, 1998.

\bibitem[BCC93]{BalasCC93a}
Egon Balas, Sebasti{\'a}n Ceria, and G{\'e}rard Cornu{\'e}jols.
\newblock A lift-and-project cutting plane algorithm for mixed {$0$}-{$1$}
  programs.
\newblock {\em Math. Programming}, 58(3, Ser. A):295--324, 1993.

\bibitem[BDG16]{BodurDG2016}
Merve Bodur, Sanjeeb Dash, and Oktay G{\"u}nl{\"u}k.
\newblock A new lift-and-project operator.
\newblock 2016.

\bibitem[BZ04]{BienstockZ04a}
Daniel Bienstock and Mark Zuckerberg.
\newblock Subset algebra lift operators for 0-1 integer programming.
\newblock {\em SIAM Journal on Optimization}, 15(1):63--95, 2004.

\bibitem[CCH89]{ChvatalCH89a}
Va{\v{s}}ek Chv{\'a}tal, William Cook, and Mark Hartmann.
\newblock On cutting-plane proofs in combinatorial optimization.
\newblock {\em Linear algebra and its applications}, 114:455--499, 1989.

\bibitem[CD01]{CookD01a}
William Cook and Sanjeeb Dash.
\newblock On the matrix-cut rank of polyhedra.
\newblock {\em Mathematics of Operations Research}, 26(1):19--30, 2001.

\bibitem[Che07]{Cheung07a}
Kevin K.~H. Cheung.
\newblock Computation of the {L}asserre ranks of some polytopes.
\newblock {\em Mathematics of Operations Research}, 32(1):88--94, 2007.

\bibitem[CL01]{CornuejolsL01a}
G{\'e}rard Cornu{\'e}jols and Yanjun Li.
\newblock On the rank of mixed 0, 1 polyhedra.
\newblock In {\em International Conference on Integer Programming and
  Combinatorial Optimization}, pages 71--77. Springer, 2001.

\bibitem[GPT10]{GouveiaPT10a}
Jo{\~a}o Gouveia, Pablo~A. Parrilo, and Rekha~R. Thomas.
\newblock Theta bodies for polynomial ideals.
\newblock {\em SIAM Journal on Optimization}, 20(4):2097--2118, 2010.

\bibitem[GT01]{GoemansT01a}
Michel~X. Goemans and Levent Tun{\c{c}}el.
\newblock When does the positive semidefiniteness constraint help in lifting
  procedures?
\newblock {\em Mathematics of Operations Research}, 26(4):796--815, 2001.

\bibitem[KLM15]{KurpiszLM15a}
Adam Kurpisz, Samuli Lepp{\"a}nen, and Monaldo Mastrolilli.
\newblock On the hardest problem formulations for the 0/1 {L}asserre hierarchy.
\newblock {\em arXiv preprint arXiv:1510.01891}, 2015.

\bibitem[KLM16]{KurpiszLM2016}
Adam Kurpisz, Samuli Lepp{\"a}nen, and Monaldo Mastrolilli.
\newblock Tight sum-of-squares lower bounds for binary polynomial optimization
  problems.
\newblock {\em arXiv preprint arXiv:1605.03019}, 2016.

\bibitem[Las01]{Lasserre01a}
Jean~B. Lasserre.
\newblock An explicit exact {SDP} relaxation for nonlinear 0-1 programs.
\newblock In {\em Integer programming and combinatorial optimization
  ({U}trecht, 2001)}, volume 2081 of {\em Lecture Notes in Comput. Sci.}, pages
  293--303. Springer, Berlin, 2001.

\bibitem[Lau03]{Laurent03a}
Monique Laurent.
\newblock A comparison of the {S}herali--{A}dams, {L}ov\'asz-{S}chrijver, and
  {L}asserre relaxations for 0-1 programming.
\newblock {\em Mathematics of Operations Research}, 28(3):470--496, 2003.

\bibitem[LS91]{LovaszS91a}
L{\'a}szl{\'o} Lov{\'a}sz and Alexander Schrijver.
\newblock Cones of matrices and set-functions and {$0$}-{$1$} optimization.
\newblock {\em SIAM Journal on Optimization}, 1(2):166--190, 1991.

\bibitem[SA90]{SheraliA90a}
Hanif~D. Sherali and Warren~P. Adams.
\newblock A hierarchy of relaxations between the continuous and convex hull
  representations for zero-one programming problems.
\newblock {\em SIAM Journal on Discrete Mathematics}, 3(3):411--430, 1990.

\bibitem[Zuc03]{Zuckerberg03a}
Mark Zuckerberg.
\newblock {\em A Set Theoretic Approach to Lifting Procedures for 0,1 Integer
  Programming}.
\newblock PhD thesis, Columbia University, 2003.

\end{thebibliography}

\end{document}